\documentclass[10pt,a4paper,listof=totoc,bibliography=totoc,listof=flat]{article}
\usepackage[latin1]{inputenc}
\usepackage{amssymb}
\usepackage{amsmath}
\usepackage{amsthm}
\newtheorem{theorem}{Theorem}
\newtheorem{lemma}{Lemma}
\newtheorem{corollary}{Corollary}
\usepackage[final]{microtype}
\usepackage{tikz}
\usepackage{subfig}
\usepackage{caption}
\usepackage{float}

\newcommand{\R}{\mathbb{R}}
\newcommand{\Z}{\mathbb{Z}}
\definecolor{polytopeColor}{gray}{0.9}

\def\Circuits{{\mathcal C}}

\DeclareMathOperator{\rank}{rank}

\def\ve#1{\mathchoice{\mbox{\boldmath$\displaystyle\bf#1$}}
{\mbox{\boldmath$\textstyle\bf#1$}}
{\mbox{\boldmath$\scriptstyle\bf#1$}}
{\mbox{\boldmath$\scriptscriptstyle\bf#1$}}}

\newcommand\veb{{\ve b}}
\newcommand\vecc{{\ve c}}
\newcommand\ved{{\ve d}}
\newcommand\vece{{\ve e}}

\newcommand\veg{{\ve g}}

\newcommand\vel{{\ve l}}

\newcommand\veu{{\ve u}}
\newcommand\vev{{\ve v}}
\newcommand\vew{{\ve w}}
\newcommand\vex{{\ve x}}
\newcommand\vey{{\ve y}}
\newcommand\vez{{\ve z}}
\newcommand\veo{{\ve 0}}
\newcommand{\CD}{\mathcal{CD}}

\newcommand{\T}{{\intercal}} 

\newcommand{\DeclareBracket}[3]{
  \newcommand{#1}[2][]{%
  \ifthenelse%
  {\equal{##1}{}}%
  {\left#2##2\right#3}%
  {\csname ##1l\endcsname#2##2\csname ##1r\endcsname#3}}}
\DeclareBracket\set\{\}

\definecolor{gray_e}{gray}{0.85}
\definecolor{gray_m}{gray}{0.9}
\definecolor{gray_f}{gray}{0.95}
\definecolor{gray_r}{gray}{0.8}
\definecolor{gray_b}{gray}{0.8}
\definecolor{gray_br}{gray}{0.75}
\definecolor{gray_s}{gray}{0.7}
  \tikzset{> = latex}

\begin{document}

\title{Edges vs Circuits: a Hierarchy of Diameters in Polyhedra}

\author{S. Borgwardt, J. De Loera, E. Finhold}

\maketitle

\begin{abstract}
The study of the graph diameter of polytopes is a classical open problem in polyhedral geometry and the theory of linear optimization. In this paper we continue the investigation initiated in \cite{bfh-14} by introducing a vast hierarchy of generalizations to the notion of graph diameter. This hierarchy provides some interesting lower bounds for the usual graph diameter. After explaining the structure of the hierarchy and discussing these bounds, we focus on clearly explaining the differences and similarities among the many diameter notions of our hierarchy. Finally, we fully characterize the hierarchy in dimension two. It collapses into fewer categories, for which we exhibit the ranges of values that can be realized as diameters.
\end{abstract}



\section{Introduction}

Dantzig's Simplex method from 1947 and its variations are the most common algorithms for solving linear programs. It can be viewed as a family of combinatorial local search algorithms on the graph of a convex polyhedron. More precisely, the search is done over the \emph{graph of the polyhedron}, which is composed of the zero- and one-dimensional faces of the feasible region (called \emph{vertices} and \emph{edges}). The search moves from a vertex of  the graph to a better neighboring vertex joined by an edge.

 The \emph{(graph) diameter} (or \emph{combinatorial diameter}) of a polyhedron is the diameter of its graph, the length of the longest shortest path among all possible pairs of vertices. Despite great effort of analysis, it remains open whether there is always a polynomial bound on the shortest path between two vertices in the graph (see for example \cite{ks-10}). While trying to understand this well-known problem, the authors of \cite{bfh-14} introduced a very natural generalization of the notion of  diameter. Here we continue their work by introducing a hierarchy of possible diameter definitions. We will see that the hierarchy includes the traditional graph diameter and the \emph{circuit diameter} introduced in \cite{bfh-14}.
 
In the following we will consider polyhedra of the general form $P(\veb,\ved)=\{\,\vez\in \R^n:A\vez=\veb,\,B\vez\leq \ved\,\}$ for matrices $A\in\Z^{m_A\times n}$, $B\in\Z^{m_B\times n}$. Note that the matrix $B$ should have full row rank $n$ for the polyhedron to have vertices and edges.
 The \emph{circuits} or \emph{elementary vectors} associated with matrices $A$ and $B$ are those vectors  $\veg\in \ker(A) \setminus \{\,\ve 0\,\}$, for which $B\veg$ is support-minimal in the set $\{\, B\vez:\vez\in \ker(A)\setminus \{\,\ve 0\,\}\,\}$. The vectors $\veg$ are always normalized to  have coprime integer components and thus, there are only finitely many such vectors. It can be shown that the set of circuits consists exactly of all edge directions of $P(\veb,\ved)$ for varying $\veb$ and $\ved$. In particular the circuits provide augmenting directions to any non-optimal solution of $\min\{\,\vecc^\T\vez:A\vez=\veb,\,B\vez\leq \ved\,\}$ for any choice of $\veb$, $\ved$ and $\vecc$. It should be noted that the circuits are
as expected related to the matroid of linear dependences of the matrix {\small $ \left( \begin{array}{cc}
A & O  \\
B & I  \end{array} \right).$} Thus by using the circuits as measurement steps for a distance we are allowing for bounds in a family of parametric polyhedra that result from
translation of defining hyperplanes.  
\smallskip

We remark that circuits have already played a fundamental role in various aspects of the theory of linear optimization (see e.g., \cite{bk-92,orientedmatroids,blandpivots,ft-97,rocka}).  Note also that for a linear program, augmentation along circuit directions is a generalization of the Simplex method: While in the Simplex method one walks only along the graph (so in particular on the boundary) of one  polyhedron for fixed $\veb,\ved$, the circuit steps could go through the interior of the polyhedron (but along \emph{potential} edge directions of other polyhedra in the same parametric family). 

Let us now define a very general notion of \emph{distance} based on circuits. 
 Let $P$ be a polyhedron and let $\Circuits$ be the set of circuits for the associated matrices $A$ and $B$. For a pair of two vertices $\vev^{(1)},\vev^{(2)}$ of $P$, we call a sequence $\vev^{(1)}=\vey^{(0)},\ldots,\vey^{(k)}=\vev^{(2)}$ a \emph{circuit walk of length $k$} if for all $i=0,\ldots,k-1$ we have
$\vey^{(i+1)}-\vey^{(i)}=\alpha_i\veg^{i}$ for some  circuit $\veg^{i}$ and some $\alpha_i>0$. Note that because we are allowing the $\alpha_i$ to be arbitrary real non-negative numbers there are
walks that can be infinite, but we restrict our attention to those that are finite and we can define:
 The \emph{circuit distance} from $\vev^{(1)}$ to $\vev^{(2)}$  is the minimum length of a circuit walk from $\vev^{(1)}$ to $\vev^{(2)}$. We call a circuit walk that realizes the circuit distance a \emph{shortest} or \emph{optimal} walk.
 The \emph{circuit diameter} of $P$ is the maximum circuit distance between any two vertices of $P$.

Our hierarchy will include different notions of circuit distances which arise by considering circuit walks that satisfy additional properties. We write $P$ for $P(\veb,\ved)$ for
fixed $\veb$ and $\ved$:
\begin{enumerate}
\item[\textbf{(e)}] If $\vey^{(i)}$ and $\vey^{(i+1)}$ are neighboring vertices in the graph of the polyhedron for all $i=0,\dots, k-1$, we call the walk an {\emph{edge walk}}. This is the term that corresponds to the classical graph diameter of a polytope.
\item[\textbf{(f)}] If $\vey^{(i)}\in P$ for all $i=0,\dots, k-1$, then we say the circuit walk is \emph{feasible}. 
\item[\textbf{(m)}] If the extension multipliers $\alpha_i$ are maximal, i.e. if $\vey^{(i)}+\alpha\veg^{i}$ is infeasible (i.e., lies outside $P$) for all $\alpha>\alpha_i$, we say that the walk is of \emph{maximum extension length} or simply \emph{maximal}. Otherwise, we say that the extension is of \emph{arbitrary length}. 
\item[\textbf{(r)}] If no circuit is repeated, then we say the walk is \emph{non-repetitive}.
\item[\textbf{(b)}] If no pair of circuits $\veg^{i},-\veg^{i}$ is used, then we say the walk is \emph{non-backwards}.
\item[\textbf{(s)}]  Two vectors $\vex$ and $\vey$ are sign-compatible with respect to the matrix $B$ defining the polyhedron $P=\{\,\vez\in \R^n:A\vez=\veb,\,B\vez\leq \ved\,\}$, if $B\vex$ and $B\vey$ belong to the same orthant of $\R^{m_B}$, that is, their $i$-th components $(B\vex)_i$ and $(B\vey)_i$ satisfy $(B\vex)_i\cdot (B\vey)_i \geq 0$ for all $i=1,\ldots,m_B$.
If all the circuits are pairwise sign-compatible and are sign-compatible with the vector $\vev^{(2)}-\vev^{(1)}$, we say the walk is \emph{sign-compatible}. 
\end{enumerate}

In what follows, we consider circuit distances restricted to different combinations of these properties and relate them to each other. A prime example would be the following: In the Simplex method one is limiting augmentation directions to actual edge directions at the current vertex and always choosing maximal augmentations to another vertex. In particular one ensures that the next point on the walk is feasible. Hence such walks satisfy the properties \textbf{(e)}, \textbf{(f)} and \textbf{(m)}. For several of the distance concepts we present, we liberate ourselves from some of these restrictions: We try to go from $\vev^{(1)}$ to $\vev^{(2)}$ more efficiently by possibly going through the interior of the polyhedron along linear combinations of circuits. We are even willing to leave the feasible region if that may yield fewer steps. Figure \ref{fig:typesofrestriction} depicts some walks for different combinations of these properties.

	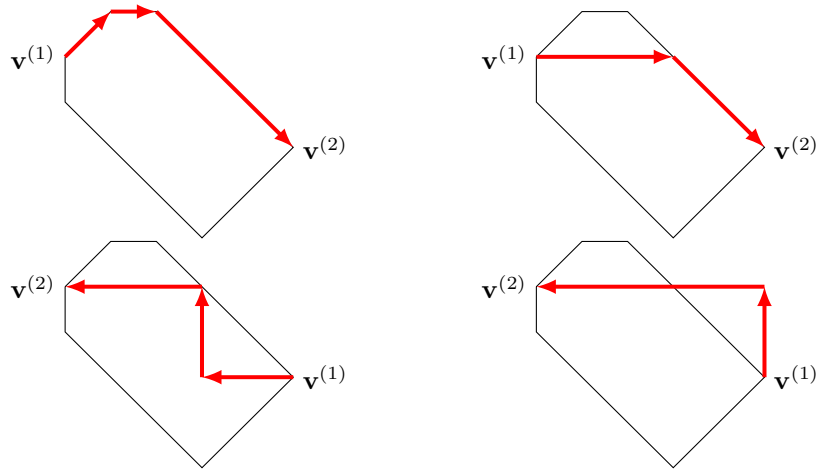
\begin{figure}[H]
		\centering
			\begin{tikzpicture}[scale=0.6]
					\coordinate (v1) at (0,1);
					\coordinate (v2) at (1,2);
					\coordinate (v3) at (2,2);
					\coordinate (v4) at (5,-1);
					\coordinate (v5) at (3,-3);
					\coordinate (v6) at (0,0);
					\draw[black] (v1)--(v2)--(v3)--(v4)--(v5)--(v6)--(v1);					
					\node[left] at (v1) {$\vev^{(1)}$};
					\node[right] at (v4) {$\vev^{(2)}$};
					\draw[line width= 1.5, red, ->] (v1)--(v2);
					\draw[line width= 1.5, red, ->] (v2)--(v3);
					\draw[line width= 1.5, red, ->] (v3)--(v4);
			\end{tikzpicture}
			\qquad \qquad
			\begin{tikzpicture}[scale=0.6]
					\coordinate (v1) at (0,1);
					\coordinate (v2) at (1,2);
					\coordinate (v3) at (2,2);
					\coordinate (v4) at (5,-1);
					\coordinate (v5) at (3,-3);
					\coordinate (v6) at (0,0);
					\draw[black] (v1)--(v2)--(v3)--(v4)--(v5)--(v6)--(v1);					
					\node[left] at (v1) {$\vev^{(1)}$};
					\node[right] at (v4) {$\vev^{(2)}$};
					\draw[line width= 1.5, red, ->] (v1)--(3,1);
					\draw[line width= 1.5, red, ->] (3,1)--(v4);
			\end{tikzpicture}
\qquad \qquad
			\begin{tikzpicture}[scale=0.6]
					\coordinate (v1) at (0,1);
					\coordinate (v2) at (1,2);
					\coordinate (v3) at (2,2);
					\coordinate (v4) at (5,-1);
					\coordinate (v5) at (3,-3);
					\coordinate (v6) at (0,0);
					\draw[black] (v1)--(v2)--(v3)--(v4)--(v5)--(v6)--(v1);					
					\node[left] at (v1) {$\vev^{(2)}$};
					\node[right] at (v4) {$\vev^{(1)}$};
					\draw[line width= 1.5, red, <-] (v1)--(3,1);
					\draw[line width= 1.5, red, <-] (3,1)--(3,-1);
					\draw[line width= 1.5, red, <-] (3,-1)--(v4);
			\end{tikzpicture}
\qquad \qquad
			\begin{tikzpicture}[scale=0.6]
					\coordinate (v1) at (0,1);
					\coordinate (v2) at (1,2);
					\coordinate (v3) at (2,2);
					\coordinate (v4) at (5,-1);
					\coordinate (v5) at (3,-3);
					\coordinate (v6) at (0,0);
					\draw[black] (v1)--(v2)--(v3)--(v4)--(v5)--(v6)--(v1);					
					\node[left] at (v1) {$\vev^{(2)}$};
					\node[right] at (v4) {$\vev^{(1)}$};
					\draw[line width= 1.5, red, <-] (v1)--(5,1);
					\draw[line width= 1.5, red, <-] (5,1)--(v4);
			\end{tikzpicture}
	\caption{An edge walk and a feasible maximal walk (first row). A feasible (repetitive) walk and an unrestricted walk (second row).}
	\end{figure}\label{fig:typesofrestriction}

We now introduce a uniform notation for our discussion. We use $\mathcal{CD}$ to refer to the circuit distance from $\vev^{(1)}$ to $\vev^{(2)}$ with no further restrictions. When considering only circuit walks on which we impose some of the above restrictions, we denote these restrictions by small subscript letters as used in the above list of properties. For example $\mathcal{CD}_{{f}s}$ refers to the feasible sign-compatible circuit distance, where the corresponding walk is feasible and sign-compatible, while $\mathcal{CD}_{{f}mr}$ means we have to use a feasible, maximal and non-repetitive walk. To have a simple wording, we call, for example, $\CD_{fm}$ the feasible maximal circuit distance and do the same for all other circuit distances. In addition, we here give explicit names to the four circuit distances that will form the core of our hierarchy:

Note that $\CD_{efm}$ is the classical \emph{graph distance} in the polytope $P$, while $\CD_{fm}$ corresponds to the original \emph{circuit distance} as introduced in \cite{bfh-14}. Further, we call $\CD_{f}$ the \emph{weak circuit distance} and $\CD$ the \emph{soft circuit distance}.  As we often have to carefully distinguish different types of circuit distance, we stick to identifying them by their properties in many cases, but these four distances are the most fundamental in our work (see Theorem \ref{thm:hierarchy} and the
central column of Figure \ref{Fig: awesome circuit hierarchy}).

Why are these distances interesting? First, note the graph diameter is bounded (below) by diameters that have much weaker properties and that therefore may be much easier to bound or to compute. Second, we will show the different diameters shed some light on bounding the graph diameter. For some polytopes the differences are large but in others they are not (e.g., in \cite{bdlf-14} we show there are only small differences for transportation polytopes). 
Many pairs of these circuit distances have easy-to-verify relationships to each other. For two given vertices, e.g{.} the weak circuit distance $\mathcal{CD}_f$ is at least as large as the soft circuit distance $\mathcal{CD}$ because we are just imposing an additional constraint. We denote this $\mathcal{CD}_f \geq \mathcal{CD}$. If there are polyhedra with vertices such that these two values differ, we write $\mathcal{CD}_f  > \mathcal{CD}$. Sometimes we will consider several such combinations at the same time. We then e.g.{} use $\mathcal{CD}_{f(s)} > \mathcal{CD}_{(s)}$ to refer to both $\mathcal{CD}_f > \mathcal{CD}$ and $\mathcal{CD}_{fs}> \mathcal{CD}_s$. 
Note that this notation is transitive: Clearly $\CD_{fm}\geq \CD_f \geq \CD$ implies $\CD_{fm} \geq \CD$ and  $\CD_{fm}> \CD_f \geq \CD$ implies $\CD_{fm} > \CD$. 
The main goal of this paper is to prove inequalities between the different distances, show how they strictly or weakly bound each other, and show how they differ.

Some general comments are in order before we list our results. First, as we will see later, some of optimal walks  are \emph{commutative} in the sense that it does not matter in which order we apply the steps.  This happens for the diameters $\CD$ and $\CD_{fs}$. Such commutative walks can be interpreted simply as linear combinations of circuits of the form $\vev^{(2)}-\vev^{(1)}=\sum\limits_{i=1}^k \alpha_i \veg^i$. All other types of circuit walks have to be regarded as \emph{ordered} sequences of vectors. In this way, the distance $\CD$ is just a linear algebra bound of the graph diameter that equals the size of a minimal support of a linear combination of circuits. 

Second, it is important to note that reversing the walk from $\vev^{(1)}$ to $\vev^{(2)}$ (by taking the negatives of circuits) gives a walk from $\vev^{(2)}$ to $\vev^{(1)}$, but this new walk may not necessarily satisfy the same properties. See \cite{bfh-14} for a simple counterexample with respect to $\CD_{fm}$. However, fortunately all of the distance concepts besides $\CD_{fm(b)(r)}$, are \emph{symmetric} in the sense that the reversed walk satisfies the conditions the original walk did and thus the distance from $\vev^{(2)}$ to $\vev^{(1)}$ is the same as the distance from $\vev^{(1)}$ to $\vev^{(2)}$.
Finally, sign-compatible walks may not be obviously natural for the non-expert, but it was shown in \cite{Hemmecke+Onn+Weismantel:oracle} they play a significant role in showing that there is a selection strategy such that only polynomially many circuit greedy-like augmentation steps that respect sign-compatibility are needed to reach an optimal linear programming solution (a fact that is still unresolved for the Simplex method). However, it is still an open problem how to implement this greedy-type augmentation oracle in polynomial time.

\subsection*{Our Contributions}

Our main result is the following

\begin{theorem}\label{thm:hierarchy}
The circuit distances satisfy a hierarchy as depicted in Figure \ref{Fig: awesome circuit hierarchy}. The sign $\geq$, denotes that for any given pair of vertices one type of circuit distance always upper bounds the other. Respectively, $>$ means that one diameter strictly upper bounds the other and that there exists a polyhedron with a pair of vertices for which the two distances strictly differ.
\end{theorem}

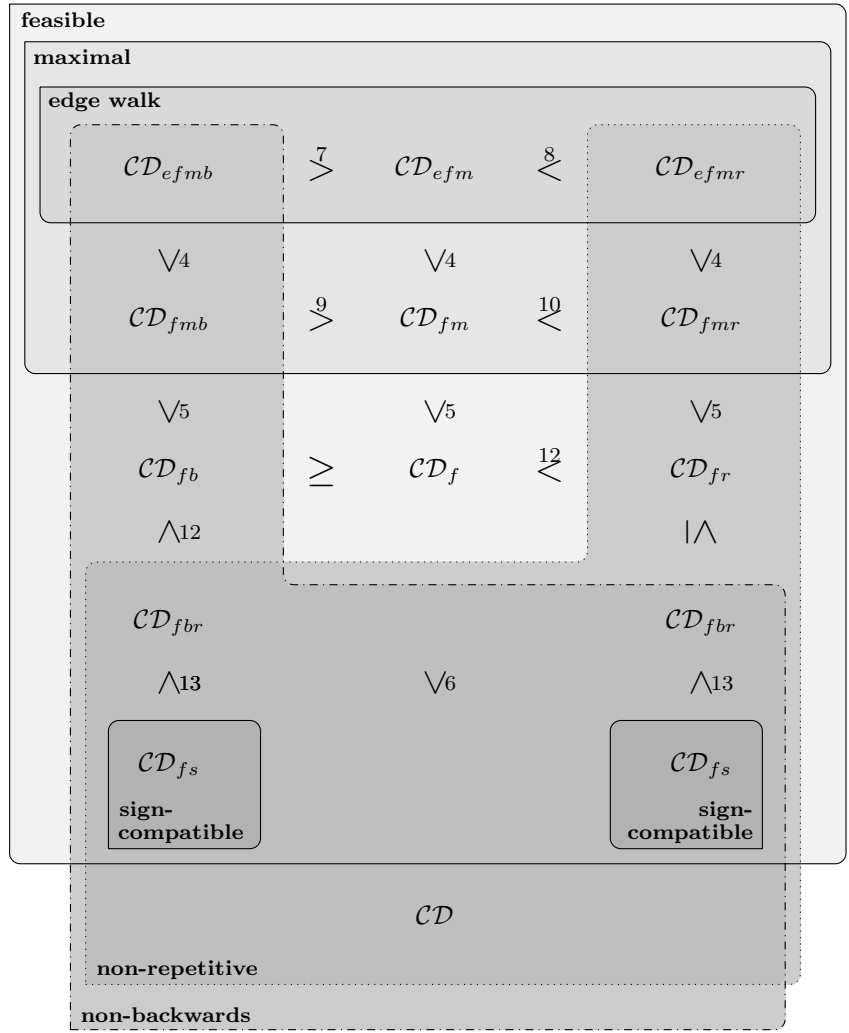
\begin{figure}
		\centering
			\begin{tikzpicture}[scale=1]
					\draw[gray_f, fill= gray_f,rounded corners] (-5.6,2.2)--(5.4,2.2)--(5.4,-9.2)--(-5.6,-9.2)--(-5.6,2.2);
					\draw[gray_m, fill= gray_m,rounded corners] (-5.4,1.7)--(5.2,1.7)--(5.2,-2.7)--(-5.4,-2.7)--(-5.4,1.7);
					\draw[gray_e, fill= gray_e,rounded corners] (-5.2,1.1)--(5,1.1)--(5,-0.7)--(-5.2,-0.7)--(-5.2,1.1);
					\draw[gray_b, fill= gray_b, rounded corners,rounded corners] (-4.8,-11.4)--(-4.8,0.6) --(-2,0.6)--(-2,-5.5)--(4.6,-5.5)--(4.6,-11.4)--(-4.8,-11.4);
					\draw[gray_r, fill= gray_r,rounded corners] (-4.6,-10.8)--(4.8,-10.8)--(4.8,0.6)--(2,0.6)--(2,-5.2)--(-4.6,-5.2)--(-4.6,-10.8);
                                               \draw[gray_br, fill= gray_br,rounded corners] (-4.6,-10.8)--(4.6,-10.8)--(4.6,-5.5)--(-2,-5.5)--(-2,-5.2)--(-4.6,-5.2)--(-4.6,-10.8);
					\draw[gray_s, fill= gray_s,rounded corners](-4.3,-9)--(-4.3,-7.3)--(-2.3,-7.3)--(-2.3,-9)--(-4.3,-9); 
					\draw[gray_s, fill= gray_s,rounded corners] (4.3,-9)--(4.3,-7.3)--(2.3,-7.3)--(2.3,-9)--(4.3,-9);

					\draw[rounded corners] (-5.6,2.2)--(5.4,2.2)--(5.4,-9.2)--(-5.6,-9.2)--(-5.6,2.2);
					\draw[rounded corners] (-5.4,1.7)--(5.2,1.7)--(5.2,-2.7)--(-5.4,-2.7)--(-5.4,1.7);
					\draw[rounded corners] (-5.2,1.1)--(5,1.1)--(5,-0.7)--(-5.2,-0.7)--(-5.2,1.1);
					\draw[dashdotted, rounded corners] (-4.8,-11.4)--(-4.8,0.6) --(-2,0.6)--(-2,-5.5)--(4.6,-5.5)--(4.6,-11.4)--(-4.8,-11.4);
					\draw[dotted, rounded corners] (-4.6,-10.8)--(4.8,-10.8)--(4.8,0.6)--(2,0.6)--(2,-5.2)--(-4.6,-5.2)--(-4.6,-10.8);
					\draw[rounded corners] (-4.3,-9)--(-4.3,-7.3)--(-2.3,-7.3)--(-2.3,-9)--(-4.3,-9); 
					\draw[rounded corners] (4.3,-9)--(4.3,-7.3)--(2.3,-7.3)--(2.3,-9)--(4.3,-9);

					\node at (-4.9,2) {\footnotesize\textbf{feasible}};
					\node at (-4.65,1.5) {\footnotesize\textbf{maximal}};
					\node at (-4.35,0.9) {\footnotesize\textbf{edge walk}};
					\node at (-3.4,-10.6) { \footnotesize  \textbf{non-repetitive}};
					\node at (-3.55,-11.2) { \footnotesize  \textbf{non-backwards}};
					\node at (-3.8,-8.5) { \footnotesize  \textbf{sign-}};
					\node at (-3.35,-8.8) { \footnotesize  \textbf{compatible}};
					\node at (3.82,-8.5) { \footnotesize  \textbf{sign-}};
					\node at (3.35,-8.8) { \footnotesize  \textbf{compatible}};					
					
					\coordinate (efmb)  node at (-3.5, 0) {$\CD_{efmb}$};
					\node at (-1.5,0) {\Large$>$};
					\node[above] at (-1.5,0) {\footnotesize\ref{Ex: CD0 backwards steps}};
					\coordinate (efm)  node at 	( 0, 0) {$\CD_{efm}$};
					\node at (1.5,0) {\Large$<$};
					\node[above] at (1.5,0) {\footnotesize \ref{Ex: CD0 repeats steps}};
					\coordinate (efmr)  node at ( 3.5, 0) {$\CD_{efmr}$};
					\node[rotate=-90] at (-3.5,-1.2) {\Large $>$};
					\node[rotate=-90] at (0,-1.2) {\Large $>$};
					\node[rotate=-90] at (3.5,-1.2)  {\Large $>$};
					\node[right] at (-3.5,-1.2) {\footnotesize  \ref{ex: edge-directions is stronger}};
					\node[right] at (0,-1.2) {\footnotesize \ref{ex: edge-directions is stronger}};
					\node[right] at (3.5,-1.2) {\footnotesize \ref{ex: edge-directions is stronger}};

					\coordinate (fmb)  node at 	(-3.5,-2) {$\CD_{fmb}$};
					\node at (-1.5,-2) {\Large$>$};
					\node[above] at (-1.5,-2) {\footnotesize \ref{ex: backwards is stronger}};
					\coordinate (fm)  node at 	( 0,-2) {$\CD_{fm}$};
					\node at (1.5,-2) {\Large$<$};
					\node[above] at (1.5,-2) {\footnotesize \ref{ex: repeated is stronger} };
					\coordinate (fmr)  node at 	( 3.5,-2) {$\CD_{fmr}$};
					\node[rotate=-90] at (-3.5,-3.2) {\Large $>$};
					\node[rotate=-90] at (3.5,-3.2)  {\Large $>$};
					\node[rotate=-90] at (0,-3.2)  {\Large $>$};
					\node[right] at (-3.5,-3.2) {\footnotesize \ref{ex: arbitrary length steps is stronger}};
					\node[right] at (3.5,-3.2) {\footnotesize \ref{ex: arbitrary length steps is stronger}};
					\node[right] at (0,-3.2) {\footnotesize \ref{ex: arbitrary length steps is stronger}};

					\coordinate (fb)  node at 	(-3.5,-4) {$\CD_{fb}$};
					\node at (-1.5,-4) {\Large$\geq$};
					\coordinate (f)  node at 		( 0,-4) {$\CD_{f}$};
					\node[rotate=-90] at (0,-6.8) {\Large $>$};
					\node at (1.5,-4) {\Large$<$};
					\node[right] at (0,-6.8) {\footnotesize \ref{Ex: infeasibility is stronger}};
					\node[above] at (1.5,-4) {\footnotesize \ref{ex: with repeated is stronger}};
					\coordinate (fr)  node at 	( 3.5,-4) {$\CD_{fr}$};

					\node[rotate=-90] at (-3.5,-4.8){\Large $<$};
					\node[right] at (-3.5,-4.8) {\footnotesize  \ref{ex: with repeated is stronger}};
					\node[rotate=-90] at (3.5,-4.8) {\Large $\leq$};
				
					\coordinate (fbr)  node at 	( -3.5,-6) {$\CD_{fbr}$};
					\node[rotate=-90] at (-3.5,-6.8) {\Large $<$};
					\node[right] at (-3.5,-6.8) {\footnotesize   \ref{Ex: sign pattern compatible is stronger}};
					\coordinate (fbr)  node at 	( 3.5,-6) {$\CD_{fbr}$};
					\node[rotate=-90] at (3.5,-6.8) {\Large $<$};
					\node[right] at (-3.5,-6.8) {\footnotesize   \ref{Ex: sign pattern compatible is stronger}};
					\coordinate (fs)  node at 	(-3.5,-7.9) {$\CD_{fs}$};
					\node[right] at (3.5,-6.8) {\footnotesize   \ref{Ex: sign pattern compatible is stronger}};
					\coordinate (fs)  node at 	(3.5,-7.9) {$\CD_{fs}$};

					\coordinate (no)  node at 	( 0,-9.9) {$\CD_{}$};
			\end{tikzpicture}
\caption{A hierarchy of circuit distances.}\label{Fig: awesome circuit hierarchy}
\end{figure}

Section \ref{sec: proofs} first presents some general properties and observations on our distances. We explain why the hierarchy contains precisely the depicted notions of circuit distances and why they satisfy the respective ``weak inequalities''. One key result of this discussion is

\begin{theorem}\label{thm:numberofsteps}
Let $P=\{\,\vez\in \R^n:A\vez=\veb,\,B\vez\leq \ved\,\}$ with $A\in\Z^{m_A\times n}$, $B\in\Z^{m_B\times n}$. be a polyhedron in $\mathbb{R}^n$. For all pairs of vertices of $P$ the distances $\CD_f$, $\CD_{fb}$, $\CD_{fr}$, $\CD_{fbr}$, and $\CD$ are bounded above by the distance $\CD_{fs}$. Moreover, all these distances are smaller or equal to 
$\min\{n-\rank(A), \rank(A)-n+m_B \}$.
\end{theorem}
 
We then perform the core part of the proof of Theorem \ref{thm:hierarchy}: We exhibit polytopes with pairs of vertices $v^{(1)}$, $v^{(2)}$ for which the length of optimal walks with the respective properties differ. We prove that almost all circuit distances in the hierarchy are indeed distinct and thus viable. Observe that the results on the circuit distances transfer to statements about the diameters of polyhedra being different too. 

 In Section \ref{sec:twodim}, we discuss the hierarchy for dimension $n=2$, denoting the circuit distances by $\CD^2_{~}$. We show that many different distance notions collapse into only a few distinct distances. The resulting hierarchy is depicted in Figure \ref{Fig: awesome circuit hierarchy dim 2} and proved in the following theorem, together with the possible distances of vertices in (two-dimensional) polygons.
 
 \begin{theorem}\label{Thm: circuit distances in dim 2} 
For $n=2$ the circuit hierarchy collapses as depicted in Figure \ref{Fig: awesome circuit hierarchy dim 2}.  

More precisely, for a polygon on $k$ vertices we obtain 
\begin{align*}
	\CD^2_{efmb}&\in \left\{1,\ldots,k-3\right\}\; (k\geq 5)\\
	\CD^2_{efm}=\CD^2_{efmr}&\in \left\{1,\ldots, \left\lfloor \frac{k}{2}\right\rfloor\right\}\\
	\CD^2_{fm(b)(r)}&\in \left\{1,\ldots, \left\lfloor \frac{k}{2}\right\rfloor\right\}\\
	\CD^2_{f}=\CD^2_{fr}=\CD^2_{fb}=\CD^2_{fbr}=\CD^2_{fs}=\CD^2_{}&\in \left\{1,2\right\}
\end{align*}
 Further there are polygons with pairs of vertices that attain the maximal distances in the ranges claimed above.
\end{theorem}
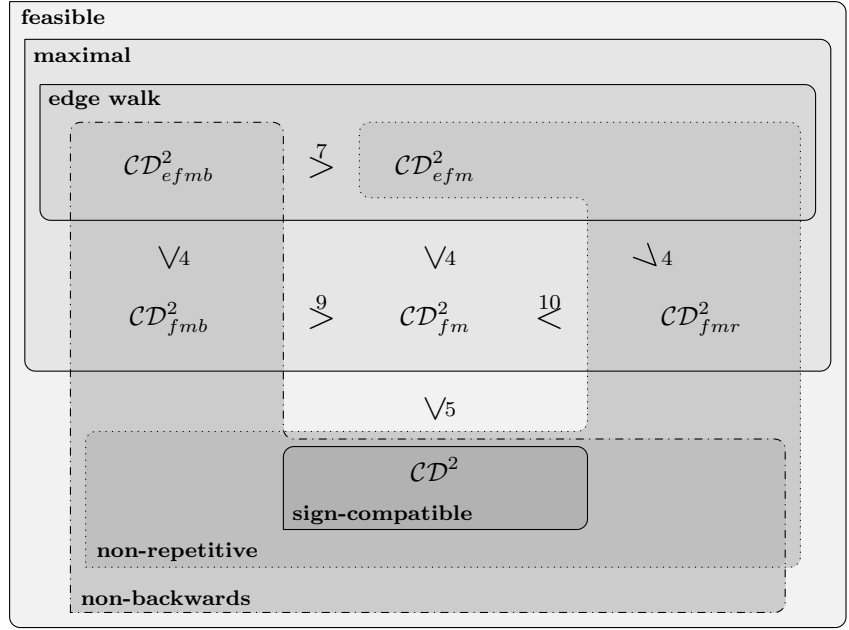
\begin{figure}[H]
		\centering
			\begin{tikzpicture}[scale=1]
					\draw[gray_f, fill= gray_f,rounded corners] (-5.6,2.2)--(5.4,2.2)--(5.4,-6.1)--(-5.6,-6.1)--(-5.6,2.2);
					\draw[gray_m, fill= gray_m,rounded corners] (-5.4,1.7)--(5.2,1.7)--(5.2,-2.7)--(-5.4,-2.7)--(-5.4,1.7);
					\draw[gray_e, fill= gray_e,rounded corners] (-5.2,1.1)--(5,1.1)--(5,-0.7)--(-5.2,-0.7)--(-5.2,1.1);
					\draw[gray_b, fill= gray_b, rounded corners,rounded corners] (-4.8,-5.9)--(-4.8,0.6) --(-2,0.6)--(-2,-3.6)--(4.6,-3.6)--(4.6,-5.9)--(-4.8,-5.9);
					\draw[gray_r, fill= gray_r,rounded corners] (-4.6,-5.3)--(4.8,-5.3)--(4.8,0.6)--(-1,0.6)--(-1,-0.4)--(2,-0.4)--(2,-3.5)--(-4.6,-3.5)--(-4.6,-5.3);
\draw[gray_br, fill= gray_br,rounded corners] (-4.6,-5.3)--(4.6,-5.3)--(4.6,-3.6)--(-2,-3.6)--(-2,-3.5)--(-4.6,-3.5)--(-4.6,-5.3);
					\draw[gray_s, fill= gray_s,rounded corners](-2,-4.8)--(-2,-3.7)--(2,-3.7)--(2,-4.8)--(-2,-4.8);

					\draw[rounded corners] (-5.6,2.2)--(5.4,2.2)--(5.4,-6.1)--(-5.6,-6.1)--(-5.6,2.2);
					\draw[rounded corners] (-5.4,1.7)--(5.2,1.7)--(5.2,-2.7)--(-5.4,-2.7)--(-5.4,1.7);
					\draw[rounded corners] (-5.2,1.1)--(5,1.1)--(5,-0.7)--(-5.2,-0.7)--(-5.2,1.1);
					\draw[dashdotted, rounded corners] (-4.8,-5.9)--(-4.8,0.6) --(-2,0.6)--(-2,-3.6)--(4.6,-3.6)--(4.6,-5.9)--(-4.8,-5.9);
					
					\draw[dotted, rounded corners] (-4.6,-5.3)--(4.8,-5.3)--(4.8,0.6)--(-1,0.6)--(-1,-0.4)--(2,-0.4)--(2,-3.5)--(-4.6,-3.5)--(-4.6,-5.3);
					\draw[rounded corners] (-2,-4.8)--(-2,-3.7)--(2,-3.7)--(2,-4.8)--(-2,-4.8); 

					\node at (-4.9,2) {\footnotesize\textbf{feasible}};
					\node at (-4.65,1.5) {\footnotesize\textbf{maximal}};
					\node at (-4.35,0.9) {\footnotesize\textbf{edge walk}};
					\node at (-3.4,-5.1) { \footnotesize  \textbf{non-repetitive}};
					\node at (-3.55,-5.7) { \footnotesize  \textbf{non-backwards}};
					\node at (-0.7,-4.6) { \footnotesize  \textbf{sign-compatible}};
					
					\coordinate (efmb)  node at (-3.5, 0) {$\CD^2_{efmb}$};
					\node at (-1.5,0) {\Large$>$};
					\node[above] at (-1.5,0) {\footnotesize\ref{Ex: CD0 backwards steps}};
					\coordinate (efm)  node at 	( 0, 0) {$\CD^2_{efm}$};
					\node[rotate=-90] at (-3.5,-1.2) {\Large $>$};
					\node[rotate=-90] at (0,-1.2) {\Large $>$};
					\node[rotate=-45] at (2.8,-1.2)  {\Large $>$};
					\node[right] at (-3.5,-1.2) {\footnotesize  \ref{ex: edge-directions is stronger}};
					\node[right] at (0,-1.2) {\footnotesize \ref{ex: edge-directions is stronger}};
					\node[right] at (2.85,-1.2) {\footnotesize \ref{ex: edge-directions is stronger}};

					\coordinate (fmb)  node at 	(-3.5,-2) {$\CD^2_{fmb}$};
					\node at (-1.5,-2) {\Large$>$};
					\node[above] at (-1.5,-2) {\footnotesize \ref{ex: backwards is stronger}};
					\coordinate (fm)  node at 	( 0,-2) {$\CD^2_{fm}$};
					\node at (1.5,-2) {\Large$<$};
					\node[above] at (1.5,-2) {\footnotesize \ref{ex: repeated is stronger} };
					\coordinate (fmr)  node at 	( 3.5,-2) {$\CD^2_{fmr}$};
					\node[rotate=-90] at (0,-3.2)  {\Large $>$};
					\node[right] at (0,-3.2) {\footnotesize \ref{ex: arbitrary length steps is stronger}};
					\coordinate (f)  node at 		( 0,-4)
					 {$\CD^2$};
			\end{tikzpicture}
\caption{The hierarchy of circuit distances in dimension two.}\label{Fig: awesome circuit hierarchy dim 2}
\end{figure}

\section{Proof of Theorem \ref{thm:hierarchy}}\label{sec: proofs}

Before we start with the technical details of the proof of Theorem \ref{thm:hierarchy}, there are a few comments to make. Figure \ref{Fig: awesome circuit hierarchy} depicts a total of $12$ different notions of circuit distance. For sake of having a clear layout the lower left and lower right parts refer to the same classes. 

The very first horizontal layer of the table contains edge walks, which we group by a small surrounding box. An edge walk always is both feasible and maximal, so there are only combinations that contain all of these properties at the same time. We distinguish between $\CD_{efm}$ and $\CD_{efmb}, \CD_{efmr}$. By imposing an additional constraint we directly have $\CD_{efmb}\geq \CD_{efm}$ and $\CD_{efmr} \geq \CD_{efm}$, but we will show the stronger statements $\CD_{efmb}>\CD_{efm}$ and $\CD_{efmr} > \CD_{efm}$. The corresponding proofs are in Lemmas \ref{Ex: CD0 backwards steps} and \ref{Ex: CD0 repeats steps}. We generally indicate the numbers of the associated lemmas at the inequality symbols.

There is no $\CD_{efms}$, as such an edge walk is not necessarily well-defined in the sense that there is not always a sign-compatible edge walk from one vertex to another. In fact, this even holds for all feasible maximal circuit walks, which are listed in the second layer. 

\begin{lemma}
For $n=2$, there is a polytope with a pair of vertices such that there is no feasible maximal sign-compatible circuit walk from one vertex to the other one. In particular there is no feasible maximal sign-compatible edge walk. 
\end{lemma}
\begin{proof}
Consider the polytope 
	\[
			P=\left\{\ \vex \in \R^2:\ \vel\leq B\vex\leq \veu\ \right\}
	\]
defined by
	\[
		B=\left(\begin{array}{rr}
			1 & 0  \\
			1 & 1  \\
			1 & -1 \\
			1 & -2 
		\end{array}\right),\;
		\vel=\left(\begin{array}{r}
			0 \\ 0\\ -1\\ -3
		\end{array}\right),\;
		\veu=\left(\begin{array}{c}
			\infty \\  6 \\4\\ \infty
		\end{array}\right) \; .
	\]
All possible edge directions $\veg$ of $P$ are given by 
	\[
			\pm \left(\begin{array}{r} 0 \\ 1		\end{array}\right),\;
			\pm \left(\begin{array}{r} 1 \\ -1		\end{array}\right),\;
			\pm \left(\begin{array}{r} 1 \\ 1		\end{array}\right),\;
			\pm \left(\begin{array}{r} 2 \\ 1	\end{array}\right)\; ,
	\]
and the corresponding vectors $B\veg$ are
	\[
			\pm \left(\begin{array}{r} 0 \\ 1 \\ -1 \\ -2		\end{array}\right),\;
			\pm \left(\begin{array}{r} 1 \\ 0 \\ 2 \\ 3		\end{array}\right),\;
			\pm \left(\begin{array}{r} 1 \\ 2 \\ 0 \\ -1\end{array}\right),\;
			\pm \left(\begin{array}{r} 2 \\ 3 \\ 1 \\ 0 \end{array}\right)\;.
	\]
We want to perform circuit walks from $\vev^{(1)}=(2,-2)^T$ to $\vev^{(2)}=(1,2)^T$. We have $B\left(\vev^{(2)}-\vev^{(1)}\right)=(-1,3,-5,-6)^T$. The only sign-compatible circuits are $(0,1)^T$ (as $B(0,1)^T=(0,1,-1,-2)^T$) and $(-1,1)^T$ (as $B(-1,1)^T=(-1,0,2,3)^T$). But choosing direction $(0,1)^T$ as well as choosing $(-1,1)^T$ for a first feasible maximal circuit step at $\vev^{(1)}$ yields points from which we cannot reach $\vev^{(2)}$ with circuits that are sign-compatible with $\vev^{(2)}-\vev^{(1)}$.
	\begin{figure}[H]
		\centering
			\begin{tikzpicture}[scale=1]
					\coordinate (v1) at (0,1);
					\coordinate (v2) at (1,2);
					\coordinate (v3) at (3,3);
					\coordinate (v4) at (5,1);
					\coordinate (v5) at (2,-2);
					\coordinate (v6) at (0,0);
					\draw[black] (v1)--(v2)--(v3)--(v4)--(v5)--(v6)--(v1);					
					\draw [fill, black] (v1) circle (0.05cm);
					\draw [fill, black] (v2) circle (0.05cm);
					\draw [fill, black] (v3) circle (0.05cm);
					\draw [fill, black] (v4) circle (0.05cm);
					\draw [fill, black] (v5) circle (0.05cm);
					\draw [fill, black] (v6) circle (0.05cm);
					\node[below] at (v5) {$\vev^{(1)}$};
					\node[above left] at (v2) {$\vev^{(2)}$};
					\draw[line width= 1.5, red, ->] (v5)--(v6);
					\draw[line width= 1.5, red, ->] (v5)--(2,2.5);
			\end{tikzpicture}
	\caption{All feasible maximal circuit steps at $\vev^{(1)}$ that are sign-compatible with $\vev^{(2)}-\vev^{(1)}$.}	\label{fig:sign}
	\end{figure}
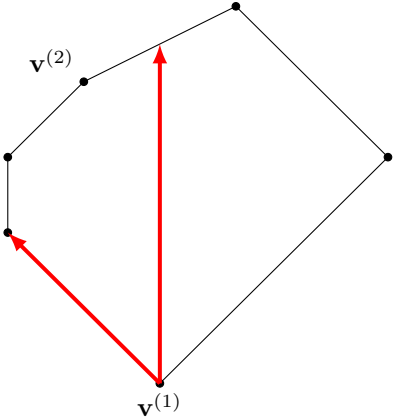
\end{proof}

In contrast, one can show that two vertices of a polyhedron $P=\{\,\vez\in \R^n:A\vez=\veb,\,B\vez\leq \ved\,\}$, are -- in our wording -- connected by a feasible sign-compatible circuit walk of boundable length. So in particular $\CD_{fs}$ is well-defined.

\begin{lemma}\label{lem:caratheodory}
Let $P=\{\,\vez\in \R^n:A\vez=\veb,\,B\vez\leq \ved\,\}$ be a polyhedron in $\mathbb{R}^n$ and let $\vev^{(1)}, \vev^{(2)}$ be two of its vertices. Then there is a feasible sign-compatible circuit walk from $\vev^{(1)}$ to $\vev^{(2)}$ of length at most $\min\{\rank(A)-n+m_B, n-\rank(A)\}$.
\end{lemma}
\begin{proof}
It suffices to consider the case $\text{rank}\left(\binom AB\right) =\text{rank}(A) + \text{rank}(B)$. Otherwise the representation of $P$ has redundant rows in the matrix $B$ and the bound derived below may only become lower.

Let $\sim \; \in \{=,\leq, \geq\}^ {m_B}$ such that its $i$-th component $\sim_i$ is defined as $$\sim_i=\begin{cases}
 		\leq & \text{if } \left(B(\vev^{(2)}-\vev^{(1)}\right)_i<0\\ 
		=& \text{if } \left(B(\vev^{(2)}-\vev^{(1)}\right)_i=0\\ 
 		\geq & \text{if } \left(B(\vev^{(2)}-\vev^{(1)}\right)_i>0 
 \end{cases}.$$
 Then $\vev^{(2)}-\vev^{(1)}\in \left\{\vex\in \R^n: B\vex\sim 0 \right\}=: C_\sim$. This is a polyhedral rational cone in which all elements are pairwise sign-compatible. Observe that $B\vev^{(i)}\leq \ved$ and at least $n-\rank(A)$ linear independent inequalities of this kind are tight. Hence $(B(\vev^{(2)}-\vev^{(1)}))_i=0$ for at least $2(n-\rank(A))-m_B$ (linearly independent) inequalities if  $m_B\leq 2(n-\rank(A))$ (and possibly for none if $m_B\geq 2(n-\rank(A))$). Hence $C_\sim$ has  dimension at most $n-(2(n-\rank(A))-m_B)=2\rank(A)-n+m_B$ if $m_B\leq 2(n-\rank(A))$ (and dimension at most $n$ if $m_B\geq 2(n-\rank(A))$). 

 Let $H_\sim$ be the unique minimal generating set of $C_\sim\cap \ker(A)$ over $\R$, where the components of each vector in $H_\sim$ are scaled to integers with greatest common divisor one. Then all elements in $H_\sim$ are circuits. Note that $\text{dim}(C_\sim\cap \ker(A))= \dim(C_\sim)-\text{rank}(A)$ due to $\text{rank}\left(\binom AB\right) =\text{rank}(A) + \text{rank}(B)$ and hence 
 $\text{dim}(C_\sim\cap \ker(A)) \leq  \rank(A)-n+m_B$ if $m_B\leq 2(n-\rank(A))$  (and $\text{dim}(C_\sim\cap \ker(A)) \leq n-\rank(A)$ if $m_B\geq 2(n-\rank(A))$). By Caratheodory's Theorem $\vev^{(2)}-\vev^{(1)}\in C_\sim\cap \ker(A)$ can thus be written as a combination of at most $\rank(A)-n+m_B$ (respectively $n-\rank(A)$) of the generators contained in $H_\sim$.
\end{proof}

By transitivity of the inequalities in the hierarchy, this upper bound on $\CD_{fs}$ transfers to many of the distances.  This proves Theorem \ref{thm:numberofsteps}.
 
Let us add what the given bound looks like for two widely-used types of polyhedra in whose description the matrix $A$, respectively $B$ does not appear.


\begin{corollary}
Let $P=\{\,\vez\in \R^n:A\vez=\veb,\, \vez \geq \veo \,\}$ be a polyhedron in $\mathbb{R}^n$ and let $\vev^{(1)}, \vev^{(2)}$ be two of its vertices. Then there is a feasible sign-compatible circuit walk from $\vev^{(1)}$ to $\vev^{(2)}$ of length at most $\min\{\rank(A), n-\rank(A)\}$.
\end{corollary}
\begin{proof}
Note that $B=-I_n$. The claim then follows from Lemma \ref{lem:caratheodory} by using $m_B=n$ .
\end{proof}

\begin{corollary}
Let $P=\{\,\vez\in \R^n:\,B\vez\leq \ved\,\}$ be a polyhedron in $\mathbb{R}^n$ and let $\vev^{(1)}, \vev^{(2)}$ be two of its vertices. Then there is a feasible sign-compatible circuit walk from $\vev^{(1)}$ to $\vev^{(2)}$ of length at most $\min\{m_B-n, n\}$.
\end{corollary}

\begin{proof}
There is no matrix $A$ in the description of the polyhedron, so the claim follows from Lemma \ref{lem:caratheodory} by using $\rank(A)=0$.
\end{proof}

We now relax the constraint \textbf{(e)} and allow circuit walks through the interior of the polyhedron.
For feasible maximal circuit walks we again distinguish between $\CD_{fm}$ and  $\CD_{fmb}, \CD_{fmr}$ and we prove that these concepts do not coincide in Lemmas \ref{ex: backwards is stronger} and \ref{ex: repeated is stronger}, that is, $\CD_{fm}<\CD_{fmb}$ and $\CD_{fm}<\CD_{fmr}$. Additionally, we show that the second layer connects to the first one not only by the obvious weak inequalities, but by $\CD_{efm(b)(r)}>\CD_{fm(b)(r)}$ in Lemma \ref{ex: edge-directions is stronger}, using a polytope from \cite{bfh-14}.

In the third and lower layers of the table, we drop the maximality condition. This may again reduce the distance of vertices, which we show in Lemma \ref{ex: arbitrary length steps is stronger}. We further prove that requiring a non-repetitive walk may increase the distance of a feasible walk, i.e{.} $\CD_{fr} > \CD_{f}$ and  $\CD_{fbr} > \CD_{fb}$ by constructing a special four-dimensional polytope in Lemma \ref{ex: with repeated is stronger}. In contrast we only know $\CD_{fb} \geq \CD_{f}$ and $\CD_{fbr} \geq \CD_{fr}$ when the non-backwards restriction is dropped. These are the only weak inequalities in the hierarchy, and we conjecture that these are strict as well. In Lemma \ref{ex: 5dim}, we explain why a polytope proving this conjecture has to be of dimension five or higher. 

We conclude the feasible circuit walks with sign-compatible ones, i.e{.} $\CD_{fs}$. Unlike the many combinations where a weak inequality is clear from imposing additional or less constraints, it is not obvious for $\CD_{fs}\geq \CD_{fbr}$.
\begin{lemma}\label{lem:smeansbr}
Any optimal sign-compatible circuit walk is in fact feasible, non-backwards and non-repetitive.
\end{lemma}
\begin{proof}
It follows immediately from the definition of sign-compatible walks, that these are feasible. In fact, the steps of a sign-compatible circuit walk can be applied in arbitrary order, yielding feasible sign-compatible walks again. Hence by reordering, we can assume that all steps that use a circuit $\pm \veg^i$ are applied consecutively. Thus these multiple steps could be combined into a single circuit step which yields a shorter circuit walk.
\end{proof}

Later we will prove in Lemma \ref{Ex: sign pattern compatible is stronger} that even $\CD_{fs}> \CD_{fbr}$ holds.

In the final part of the hierarchy, shown in the lowest horizontal layer of the table, we do not even require feasibility. This can indeed be an advantage as demonstrated in Lemma \ref{Ex: infeasibility is stronger}. Here we only have to consider $\CD$.  Lemma \ref{lem:smeansbr} tells us that every optimal sign-compatible walk is feasible, hence $\CD_s=\CD_{fs}$, and similar arguments  show that optimal soft circuit walks are non-repetitive and non-backwards, that is $\CD=\CD_r=\CD_b$.


Let us point out that there are classes of polyhedra for which the whole hierarchy `collapses'. For example, in simplices all pairs of vertices are connected by an edge, so all circuit diameters equal one. For any $n$-dimensional zonotope, all circuit diameters are equal to $n$;  the $n$-dimensional cube is a particularly simple special case. Recall that a zonotope is point-symmetric with respect to its center of gravity. Vertices that correspond to each other with respect to the point symmetry are connected by an edge walk of length exactly $n$. Using any set of circuits and no restrictions on the walk we cannot do any better, as the circuits here correspond to the actual, existing edge directions.

Finally we turn to the proofs for the `strict inequalities' in our hierarchy. We begin with the relation of edge walks and feasible maximal circuit walks. 

\begin{lemma}[$\CD_{efm(b)(r)}>\CD_{fm(b)(r)}$] \label{ex: edge-directions is stronger}
For $n=2$, there is a polytope with a pair of vertices for which every optimal feasible maximal circuit walk is not an edge walk, and there is such a walk that is non-repetitive and non-backwards. Hence the distances $\CD_{efm}$ and $\CD_{fm}$, the distances $\CD_{efmb}$ and $\CD_{fmb}$, and the distances $\CD_{efmr}$ and $\CD_{fmr}$ differ in this case.
\end{lemma}

\begin{proof}
In the polytope below, an optimal edge walk from $\vev^{(1)}$ to $\vev^{(2)}$ along the edges has length three, while there is a feasible maximal non-repetitive non-backwards circuit walk of length two.
	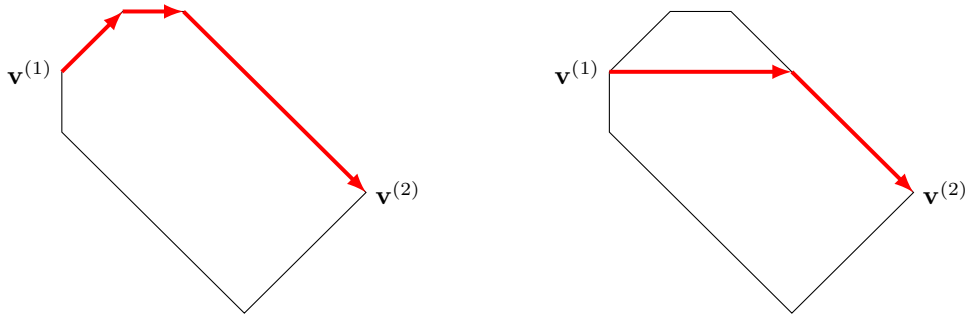
\begin{figure}[H]
		\centering
			\begin{tikzpicture}[scale=0.8]
					\coordinate (v1) at (0,1);
					\coordinate (v2) at (1,2);
					\coordinate (v3) at (2,2);
					\coordinate (v4) at (5,-1);
					\coordinate (v5) at (3,-3);
					\coordinate (v6) at (0,0);
					\draw[black] (v1)--(v2)--(v3)--(v4)--(v5)--(v6)--(v1);					
					\node[left] at (v1) {$\vev^{(1)}$};
					\node[right] at (v4) {$\vev^{(2)}$};
					\draw[line width= 1.5, red, ->] (v1)--(v2);
					\draw[line width= 1.5, red, ->] (v2)--(v3);
					\draw[line width= 1.5, red, ->] (v3)--(v4);
			\end{tikzpicture}
			\qquad \qquad
			\begin{tikzpicture}[scale=0.8]
					\coordinate (v1) at (0,1);
					\coordinate (v2) at (1,2);
					\coordinate (v3) at (2,2);
					\coordinate (v4) at (5,-1);
					\coordinate (v5) at (3,-3);
					\coordinate (v6) at (0,0);
					\draw[black] (v1)--(v2)--(v3)--(v4)--(v5)--(v6)--(v1);					
					\node[left] at (v1) {$\vev^{(1)}$};
					\node[right] at (v4) {$\vev^{(2)}$};
					\draw[line width= 1.5, red, ->] (v1)--(3,1);
					\draw[line width= 1.5, red, ->] (3,1)--(v4);
			\end{tikzpicture}
	\caption{An optimal edge walk and an optimal feasible maximal walk.}
	\end{figure}
\end{proof}

Next we turn to dropping maximality of a feasible circuit walk.

\begin{lemma}[$\CD_{fm(b)(r)}>\CD_{f(b)(r)}$]\label{ex: arbitrary length steps is stronger}
For $n=2$, there is a polytope with a pair of vertices for which every optimal feasible circuit walk is not maximal, and there is such a walk that is non-repetitive and non-backwards. Hence the distances $\CD_{fm}$ and $\CD_{f}$, the distances $\CD_{fmb}$ and $\CD_{fb}$, and the distances $\CD_{fmr}$ and $\CD_{fr}$ differ in this case.
\end{lemma}

\begin{proof} 
In the polytope below, an optimal feasible maximal walk from $\vev^{(2)}$ to $\vev^{(1)}$ has length at least three: No matter which circuit direction we apply at $\vev^{(2)}$ with maximum length, we cannot get to $\vev^{(1)}$ in just one additional step. 
	\begin{figure}[H]
		\centering
			\begin{tikzpicture}[scale=0.8]
					\coordinate (v1) at (0,1);
					\coordinate (v2) at (1,2);
					\coordinate (v3) at (2,2);
					\coordinate (v4) at (5,-1);
					\coordinate (v5) at (3,-3);
					\coordinate (v6) at (0,0);
					\draw[black] (v1)--(v2)--(v3)--(v4)--(v5)--(v6)--(v1);					
					\node[left] at (v1) {$\vev^{(1)}$};
					\node[right] at (v4) {$\vev^{(2)}$};
					\draw[line width= 1.5, red, ->] (v4)--(v3);
					\draw[line width= 1.5, red, ->] (v4)--(1,-1);
					\draw[line width= 1.5, red, ->] (v4)--(v5);
			\end{tikzpicture}
		\caption{Possible first feasible maximal circuit steps at $\vev^{(1)}$.}
	\end{figure}
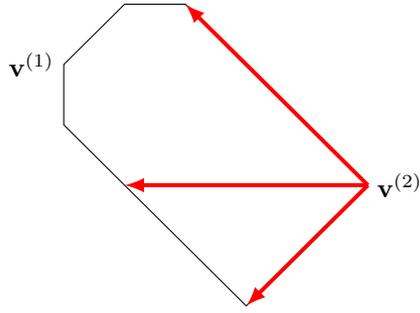
On the other hand, there is a feasible non-repetitive non-backwards circuit walk of length two.	
	\begin{figure}[H]
		\centering
			\begin{tikzpicture}[scale=0.8]
					\coordinate (v1) at (0,1);
					\coordinate (v2) at (1,2);
					\coordinate (v3) at (2,2);
					\coordinate (v4) at (5,-1);
					\coordinate (v5) at (3,-3);
					\coordinate (v6) at (0,0);
					\draw[black] (v1)--(v2)--(v3)--(v4)--(v5)--(v6)--(v1);					
					\node[left] at (v1) {$\vev^{(1)}$};
					\node[right] at (v4) {$\vev^{(2)}$};
					\draw[line width= 1.5, red, ->] (v4)--(3,1);
					\draw[line width= 1.5, red, ->] (3,1)--(v1);
			\end{tikzpicture}
	\caption{A feasible circuit walk of length two.}
	\end{figure}
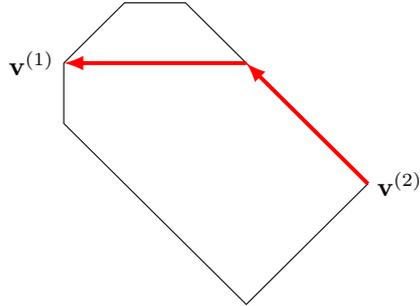
\end{proof}

We now show that a soft circuit walk may be shorter than an optimal feasible circuit walk.

\begin{lemma}[$\CD_{f}>\CD_{}$] \label{Ex: infeasibility is stronger} 
For $n=3$, there is a polytope with a pair of vertices for which no optimal circuit walk with respect to $\CD$ is feasible, and there is such an optimal walk that is sign-compatible. Hence the distance $\CD_{f}$ and $\CD_{}$ differ in this case.
\end{lemma}

\begin{proof}
The polytope below is obtained from a cube by cutting off six of its vertices using three pairs of hyperplanes, and keeping an opposite pair $\vev^{(1)},\vev^{(2)}$ of vertices as depicted. Assume the center of gravity of the cube is $\veo$. Then the normals of these hyperplanes are equal to the coordinates of the vertices cut off. The `depth' of the cuts is arbitrarily small.


Any feasible circuit walk from $\vev^{(1)}$ to $\vev^{(2)}$ has length at least three: To see this we illustrate the directions of all possible first steps at $\vev^{(1)}$ (red) and all possible last steps to $\vev^{(2)}$ (green) of a feasible circuit walk. Note that these steps are not necessarily maximal.
	\begin{figure}[H]
		\centering
			\begin{tikzpicture}[scale=0.3]
					\coordinate (v1) at (0,0);
					\coordinate (v21) at (8,0);
					\coordinate (v22) at (10,2);
					\coordinate (v23) at (11,1);
					\coordinate (v3) at (15,5);
					\coordinate (v31) at (14,4);
					\coordinate (v32) at (15,7);
					\coordinate (v33) at (13,5);
					\coordinate (v4) at (5,5);
					\coordinate (v41) at (5,7);
					\coordinate (v42) at (7,5);
					\coordinate (v43) at (4,4);
					\coordinate (v51) at (0,8);
					\coordinate (v52) at (2,10);
					\coordinate (v53) at (1,11);
					\coordinate (v61) at (8,10);
					\coordinate (v62) at (10,8);
					\coordinate (v63) at (11,11);
					\coordinate (v7) at (15,15);
					\coordinate (v81) at (7,15);
					\coordinate (v82) at (5,13);
					\coordinate (v83) at (4,14);
					\node[below left] at (v1) {$\vev^{(1)}$};
					\node[above right] at (v7) {$\vev^{(2)}$};				
	
					\draw (v62)--(v22);
					\draw (v31)--(v23);
					\draw (v63)--(v7);
					\draw (v32)--(v7);
					\draw[dashed] (v33)--(v42);
					\draw (v52)--(v61);
					\draw (v53)--(v83);
					\draw (v81)--(v7);
					\draw[dashed] (v82)--(v41);
					\draw (v21)--(v22)--(v23)--(v21);
					\draw[dashed] (v32)--(v33)--(v31);
					\draw (v31)--(v32);
					\draw[dashed] (v41)--(v42)--(v43)--(v41);
					\draw (v51)--(v52)--(v53)--(v51);
					\draw (v61)--(v62)--(v63)--(v61);
					\draw[dashed] (v81)--(v82)--(v83);
					\draw (v83)--(v81);
					\draw[line width= 1.5, red, ->] (v1)--(v51);
					\draw[line width= 1.5, red, ->] (v1)--(v21);
					\draw[line width= 1.5, red, ->, dashed] (v1)--(v43);		
					
					\draw[line width= 1.5, green, <-] (v7)--(v81);
					\draw[line width= 1.5, green, <-] (v7)--(v63);
					\draw[line width= 1.5, green, <-] (v7)--(v32);
			\end{tikzpicture}
			\qquad \qquad
				\begin{tikzpicture}[scale=0.3]
					\coordinate (v1) at (0,0);
					\coordinate (v21) at (8,0);
					\coordinate (v22) at (10,2);
					\coordinate (v23) at (11,1);
					\coordinate (v3) at (15,5);
					\coordinate (v31) at (14,4);
					\coordinate (v32) at (15,7);
					\coordinate (v33) at (13,5);
					\coordinate (v4) at (5,5);
					\coordinate (v41) at (5,7);
					\coordinate (v42) at (7,5);
					\coordinate (v43) at (4,4);
					\coordinate (v51) at (0,8);
					\coordinate (v52) at (2,10);
					\coordinate (v53) at (1,11);
					\coordinate (v61) at (8,10);
					\coordinate (v62) at (10,8);
					\coordinate (v63) at (11,11);
					\coordinate (v7) at (15,15);
					\coordinate (v81) at (7,15);
					\coordinate (v82) at (5,13);
					\coordinate (v83) at (4,14);
					\node[below left] at (v1) {$\vev^{(1)}$};
					\node[above right] at (v7) {$\vev^{(2)}$};				
					
					\draw[line width= 1.5, red, ->, dashed] (v1)--(13.5,4.5);
					\draw[line width= 1.5, red, ->, dashed] (v1)--(4.5,13.5);		
					\draw[line width= 1.5, green, <-, dashed] (v7)--(5.85,6.15);				

					\draw (v1)--(v51);
					\draw (v1)--(v21);
					\draw (v1)--(v43);						
					\draw (v62)--(v22);
					\draw (v31)--(v23);
					\draw (v63)--(v7);
					\draw (v32)--(v7);
					\draw[dashed] (v33)--(v42);
					\draw (v52)--(v61);
					\draw (v53)--(v83);
					\draw (v81)--(v7);
					\draw[dashed] (v82)--(v41);
					\draw (v21)--(v22)--(v23)--(v21);
					\draw[dashed] (v32)--(v33)--(v31);
					\draw (v31)--(v32);
					\draw[dashed] (v41)--(v42)--(v43)--(v41);
					\draw (v51)--(v52)--(v53)--(v51);
					\draw (v61)--(v62)--(v63)--(v61);
					\draw[dashed] (v81)--(v82)--(v83);
					\draw (v83)--(v81);
					\draw[line width= 1.5, red, ->] (v1)--(9.15,8.85);
					\draw[line width= 1.5, green, <-] (v7)--(10.5,1.5);
					\draw[line width= 1.5, green, <-] (v7)--(1.5,10.5);
			\end{tikzpicture}
	\caption{Possible first and last steps of a circuit walk from $\vev^{(1)}$ to $\vev^{(2)}$.}
	\end{figure}
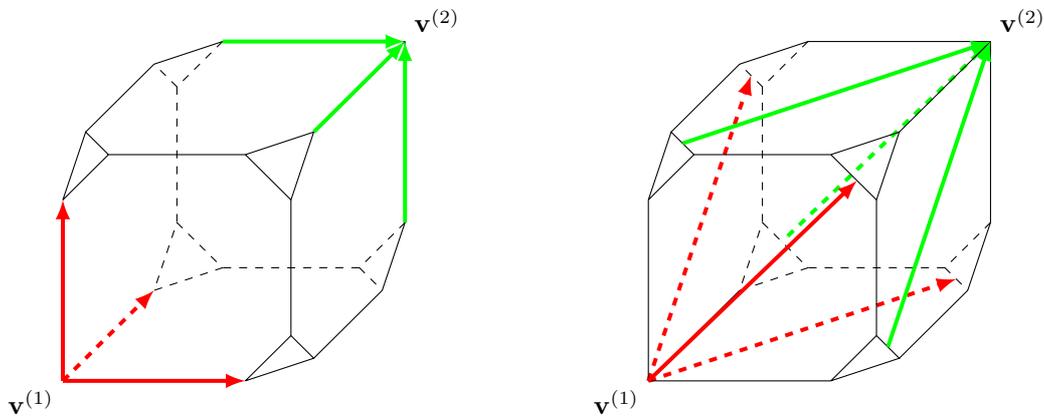
Clearly there is no point that (a)  can be reached in a single step from $\vev^{(1)}$ and (b) from which one can reach $\vev^{(2)}$ in a single step. Hence any feasible circuit walk from $\vev^{(1)}$ to $\vev^{(2)}$ has length at least three.

On the other hand, there is a soft circuit walk of length two from $\vev^{(1)}$ to $\vev^{(2)}$.
	\begin{figure}[H]
		\centering
			\begin{tikzpicture}[scale=0.3]
					\coordinate (v1) at (0,0);
					\coordinate (v21) at (8,0);
					\coordinate (v22) at (10,2);
					\coordinate (v23) at (11,1);
					\coordinate (v3) at (15,5);
					\coordinate (v31) at (14,4);
					\coordinate (v32) at (15,7);
					\coordinate (v33) at (13,5);
					\coordinate (v4) at (5,5);
					\coordinate (v41) at (5,7);
					\coordinate (v42) at (7,5);
					\coordinate (v43) at (4,4);
					\coordinate (v51) at (0,8);
					\coordinate (v52) at (2,10);
					\coordinate (v53) at (1,11);
					\coordinate (v61) at (8,10);
					\coordinate (v62) at (10,8);
					\coordinate (v63) at (11,11);
					\coordinate (v7) at (15,15);
					\coordinate (v81) at (7,15);
					\coordinate (v82) at (5,13);
					\coordinate (v83) at (4,14);
					\node[below left] at (v1) {$\vev^{(1)}$};
					\node[above right] at (v7) {$\vev^{(2)}$};				
					\draw (v1)--(v21);
					\draw (v62)--(v22);
					\draw (v31)--(v23);
					\draw (v63)--(v7);
					\draw (v32)--(v7);
					\draw[dashed] (v33)--(v42);
					\draw[dashed] (v43)--(v1);				
					\draw (v1)--(v51);
					\draw (v52)--(v61);
					\draw (v53)--(v83);
					\draw (v81)--(v7);
					\draw[dashed] (v82)--(v41);
					\draw (v21)--(v22)--(v23)--(v21);
					\draw[dashed] (v32)--(v33)--(v31);
					\draw (v31)--(v32);
					\draw[dashed] (v41)--(v42)--(v43)--(v41);
					\draw (v51)--(v52)--(v53)--(v51);
					\draw (v61)--(v62)--(v63)--(v61);
					\draw[dashed] (v81)--(v82)--(v83);
					\draw (v83)--(v81);
					\draw[dotted] (v21)--(10,0);
					\draw[dotted] (v22)--(10,0);
					\draw[dotted] (v23)--(10,0);
					\draw[line width= 1.5, red, ->] (v1)--(10,0);
					\draw[line width= 1.5, red, ->] (10,0)--(v7);
			\end{tikzpicture}
		\caption{A soft circuit walk of length two.}
	\end{figure}
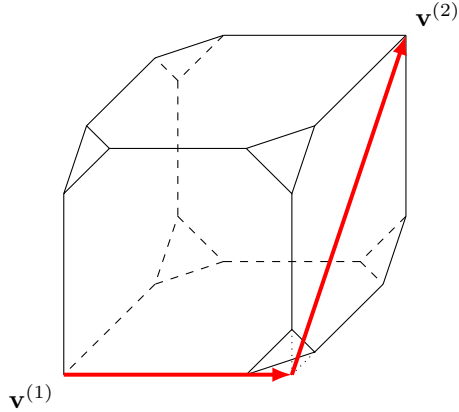
\end{proof}

The following two lemmas explain why allowing the use of edge directions both $\veg^i$ and $-\veg^i$ or the repeated use of an edge direction $\veg^i$ can yield a shorter edge walk.

\begin{lemma}[$\CD_{emfb}>\CD_{efm}$] \label{Ex: CD0 backwards steps}
For $n=2$, there is a polytope with a pair of vertices for which the unique optimal edge walk is backwards. Hence the distances $\CD_{efmb}$ and $\CD_{efm}$ differ in this case.
\end{lemma}
\begin{proof}
In the polytope below, the unique non-backwards edge walk from $\vev^{(1)}$ to $\vev^{(2)}$ has length four, while there is an edge walk of length three that uses edges in opposite directions.
	\begin{figure}[H]
		\centering
			\begin{tikzpicture}[scale=1]
					\coordinate (v1) at (0,0);
					\coordinate (v2) at (0,-3);
					\coordinate (v3) at (-1,-3);
					\coordinate (v4) at (-2,-2.5);
					\coordinate (v5) at (-2.5,-1.5);
					\coordinate (v6) at (-2,-0.5);
					\coordinate (v7) at (-1,0);
					\draw[black] (v1)--(v2)--(v3)--(v4)--(v5)--(v6)--(v7)--(v1);					
					\node[above] at (v7) {$\vev^{(1)}$};
					\node[below] at (v3) {$\vev^{(2)}$};
					\draw[line width= 1.5, red, ->] (v7)--(v6);
					\draw[line width= 1.5, red, ->] (v6)--(v5);
					\draw[line width= 1.5, red, ->] (v5)--(v4);
					\draw[line width= 1.5, red, ->] (v4)--(v3);
			\end{tikzpicture}
			\hspace{3cm}
			\begin{tikzpicture}[scale=1]
					\coordinate (v1) at (0,0);
					\coordinate (v2) at (0,-3);
					\coordinate (v3) at (-1,-3);
					\coordinate (v4) at (-2,-2.5);
					\coordinate (v5) at (-2.5,-1.5);
					\coordinate (v6) at (-2,-0.5);
					\coordinate (v7) at (-1,0);
					\draw[black] (v1)--(v2)--(v3)--(v4)--(v5)--(v6)--(v7)--(v1);					
					\node[above] at (v7) {$\vev^{(1)}$};
					\node[below] at (v3) {$\vev^{(2)}$};
					\draw[line width= 1.5, red, ->] (v7)--(v1);
					\draw[line width= 1.5, red, ->] (v1)--(v2);
					\draw[line width= 1.5, red, ->] (v2)--(v3);
			\end{tikzpicture}
	\caption{An optimal non-backwards edge walk and a backwards edge walk.}\label{Fig: CD0 backwards steps}
	\end{figure}
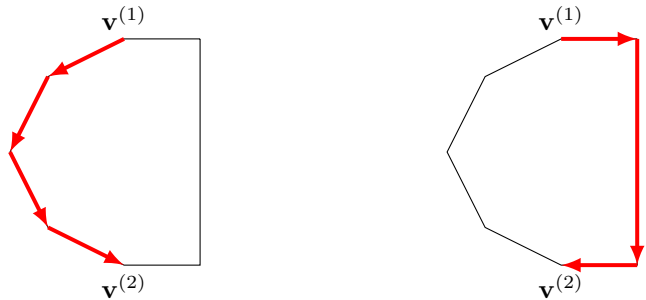
\end{proof}

\begin{lemma}[$\CD_{efmr}>\CD_{efm}$] \label{Ex: CD0 repeats steps}
For $n=3$, there is a polytope with a pair of vertices for which the unique optimal edge walk is repetitive. Hence the distances $\CD_{efmr}$ and $\CD_{efm}$ differ in this case.
\end{lemma}
\begin{proof}
We construct a polytope with the claimed property by cutting off vertices of a three-dimensional cube as illustrated in the following figures:
	\begin{figure}[H]
		\centering
			\begin{tikzpicture}[scale=0.2]
					\coordinate (v1) at (0,0);
					\coordinate (v2) at (10,0);
					\coordinate (v3) at (15,5);
					\coordinate (v4) at (5,5);
					\coordinate (v5) at (0,10);
					\coordinate (v6) at (10,10);
					\coordinate (v7) at (15,15);
					\coordinate (v8) at (5,15);
					
					\draw (v1)--(v2);
					\draw (v2)--(v3);
					\draw[dashed] (v3)--(v4);
					\draw[dashed] (v4)--(v1);
					\draw (v5)--(v6);
					\draw (v6)--(v7);
					\draw (v7)--(v8);
					\draw (v8)--(v5);
					\draw (v1)--(v5);
					\draw (v2)--(v6);
					\draw (v3)--(v7);
					\draw[dashed] (v4)--(v8);
					\draw [fill, black] (v5) circle (0.3);
					\draw [loosely dotted] (0,10)--(5,15);
			\end{tikzpicture}
			\qquad
			\begin{tikzpicture}[scale=0.2]
					\coordinate (v1) at (0,0);
					\coordinate (v2) at (10,0);
					\coordinate (v3) at (15,5);
					\coordinate (v4) at (5,5);
					\coordinate (v5) at (0,8);
					\coordinate (v6) at (10,10);
					\coordinate (v7) at (15,15);
					\coordinate (v8) at (5,15);
					
					\draw (v1)--(v2);
					\draw (v2)--(v3);
					\draw[dashed] (v3)--(v4);
					\draw[dashed] (v4)--(v1);
					\draw (v5)--(v6);
					\draw (v6)--(v7);
					\draw (v7)--(v8);
					\draw (v8)--(v5);
					\draw (v1)--(v5);
					\draw (v2)--(v6);
					\draw (v3)--(v7);
					\draw[dashed] (v4)--(v8);
					\draw (v6)--(v8);
					
					\draw [fill, black] (v4) circle (0.3);
					\draw [fill, black] (v3) circle (0.3);
					\draw [fill, black] (v7) circle (0.3);
					\draw [fill, black] (v5) circle (0.3);
			\end{tikzpicture}
			\qquad
			\begin{tikzpicture}[scale=0.2]
					\coordinate (v1) at (0,0);
					\coordinate (v2) at (10,0);
					\coordinate (v3) at (15,5);
					\coordinate (v32) at (14,4);
					\coordinate (v34) at (12,5);
					\coordinate (v37) at (15,8);
					\coordinate (v4) at (5,5);
					\coordinate (v43) at (7.5,5);
					\coordinate (v41) at (3,3);
					\coordinate (v48) at (5,7.5);
					\coordinate (v5) at (0,8);
					\coordinate (v51) at (0,6);
					\coordinate (v56) at (4.25,8.85);
					\coordinate (v58) at (2.14,11);
					\coordinate (v6) at (10,10);
					\coordinate (v7) at (15,15);
					\coordinate (v73) at (15,12);
					\coordinate (v76) at (12.5,12.5);
					\coordinate (v78) at (12,15);
					\coordinate (v8) at (5,15);
					
					\draw (v1)--(v2);
					\draw (v2)--(v32);
					\draw[dashed] (v34)--(v43);
					\draw[dashed] (v41)--(v1);
					\draw (v56)--(v6);
					\draw (v6)--(v76);
					\draw (v78)--(v8);
					\draw (v8)--(v58);
					\draw (v1)--(v51);
					\draw (v2)--(v6);
					\draw (v37)--(v73);
					\draw[dashed] (v48)--(v8);
					\draw (v6)--(v8);
					\draw[dashed] (v32)--(v34);
					\draw[dashed] (v34)--(v37);
					\draw (v37)--(v32);
					\draw[dashed] (v41)--(v43);
					\draw[dashed] (v43)--(v48);
					\draw[dashed] (v48)--(v41);
					\draw (v76)--(v78);
					\draw (v78)--(v73);
					\draw (v73)--(v76);
					\draw (v51)--(v58);
					\draw (v58)--(v56);
					\draw (v56)--(v51);
					\draw [fill, black] (v51) circle (0.3);
					\draw [fill, black] (v41) circle (0.3);
					\draw [fill, black] (v76) circle (0.3);
				\end{tikzpicture}
	\caption{Constructing the polytope by cutting of vertices (marked with dots).}
	\end{figure}
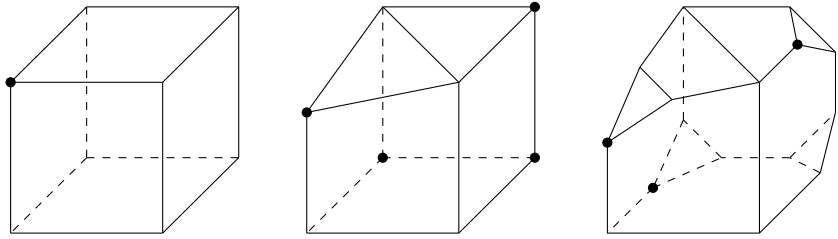
We obtain the polytope below, in which there is a repetitive edge walk from $\vev^{(1)}$ to $\vev^{(2)}$ of length four. It is easy to check that any other edge walk from $\vev^{(1)}$ to $\vev^{(2)}$ has length at least five.
	\begin{figure}[H]
		\centering
		\begin{tikzpicture}[scale=0.3]
					\coordinate (v1) at (0,0);
					\coordinate (v2) at (10,0);
					\coordinate (v3) at (15,5);
					\coordinate (v32) at (14,4);
					\coordinate (v34) at (12,5);
					\coordinate (v37) at (15,8);
					\coordinate (v4) at (5,5);
					\coordinate (v43) at (7.5,5);
					\coordinate (v41) at (3,3);
					\coordinate (v4134) at (5,4);
					\coordinate (v411) at (2.3,2.3);
					\coordinate (v4148) at (4,5.25);
					\coordinate (v48) at (5,7.5);
					\coordinate (v5) at (0,8);
					\coordinate (v51) at (0,6);
					\coordinate (v511) at (0,4.2);
					\coordinate (v5156) at (2.55,7.31);
					\coordinate (v5158) at (1.07,8.5);
					\coordinate (v56) at (4.25,8.85);
					\coordinate (v58) at (2.14,11);
					\coordinate (v6) at (10,10);
					\coordinate (v7) at (15,15);
					\coordinate (v73) at (15,12);
					\coordinate (v76) at (12.5,12.5);
					\coordinate (v766) at (11.3,11.3);
					\coordinate (v7673) at (13.5,12.3);
					\coordinate (v7678) at (12.3,13.5);
					\coordinate (v78) at (12,15);
					\coordinate (v8) at (5,15);
					
					\draw[dashed] (v4134)--(v411);
					\draw[dashed] (v411)--(v4148);
					\draw[dashed] (v4148)--(v4134);
					\draw (v766)--(v7673);
					\draw (v7673)--(v7678);
					\draw (v7678)--(v766);
					\draw (v511)--(v5156);
					\draw (v5156)--(v5158);
					\draw (v5158)--(v511);
					
					\draw (v1)--(v2);
					\draw (v2)--(v32);
					\draw[dashed] (v34)--(v43);
					\draw[dashed] (v411)--(v1);
					\draw (v56)--(v6);
					\draw (v6)--(v766);
					\draw (v78)--(v8);
					\draw (v8)--(v58);
					\draw (v1)--(v511);
					\draw (v2)--(v6);
					\draw (v37)--(v73);
					\draw[dashed] (v48)--(v8);
					\draw (v6)--(v8);
					\draw[dashed] (v32)--(v34);
					\draw[dashed] (v34)--(v37);
					\draw (v37)--(v32);
					\draw[dashed] (v4134)--(v43);
					\draw[dashed] (v43)--(v48);
					\draw[dashed] (v48)--(v4148);
					\draw (v7678)--(v78);
					\draw (v78)--(v73);
					\draw (v73)--(v7673);
					\draw (v5158)--(v58);
					\draw (v58)--(v56);
					\draw (v56)--(v5156);
					\node[below ] at (v1) {$\vev^{(1)}$};
					\node[above right] at (v78) {$\vev^{(2)}$};
					\draw[line width= 1.5, red, ->] (v1)--(v2);
					\draw[line width= 1.5, red, ->] (v2)--(v6);
					\draw[line width= 1.5, red, ->] (v6)--(v8);
					\draw[line width= 1.5, red, ->] (v8)--(v78);										
			\end{tikzpicture}
	\caption{Unique optimal edge walk from $\vev^{(1)}$ to $\vev^{(2)}$.}
	\end{figure}
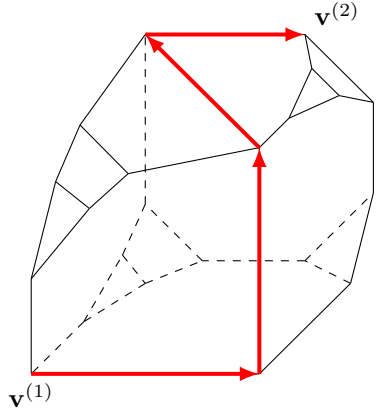
\end{proof}

Backwards or repetitive circuit walks also can be shorter than their respective counterparts. First we exhibit a polytope to see this for backwards walks.

\begin{lemma}[$\CD_{fmb}>\CD_{fm}$]\label{ex: backwards is stronger}
For $n=2$, there is a polytope with a pair of vertices for which every optimal feasible maximal circuit walk is backwards. Hence the distances $\CD_{fmb}$ and $\CD_{fm}$ differ in this case.
\end{lemma}

\begin{proof}
We consider the polytope on $11$ vertices depicted in Figure \ref{fig: backwards walk}; the lower subfigure is a zoomed-in view on the right part of the polygon. The edge directions are given by 
$$\binom 11, \binom 10, \binom{-1}{1}, \binom{3}{-10},\binom{2}{-10},\binom{1}{-10},\binom{1}{10},\binom{2}{10},\binom{3}{10}.$$There is a feasible maximal circuit walk of length three from $\vev^{(1)}$ to $\vev^{(5)}$ that is backwards.

	\begin{figure}[H]
		\centering
			\begin{tikzpicture}[scale=0.3]
					\coordinate (v1) at (19,9);
					\coordinate (v2) at (27,1);
					\coordinate (v3) at (27.24,0.2);
					\coordinate (v4) at (27.26,0.1);
					\coordinate (v5) at (27.27,0);
					\coordinate (v6) at (27.26,-0.1);
					\coordinate (v62) at (27.24,-0.2);
					\coordinate (v7) at (27,-1);
					\coordinate (v8) at (1,-1);
					\coordinate (v9) at (0,0);
					\coordinate (v10) at (9,9);
					
					\draw [fill, black] (v10) circle (0.05cm);
					\draw [fill, black] (v1) circle (0.05cm);
					\draw [fill, black] (v2) circle (0.05cm);
					\draw [fill, black] (v3) circle (0.05cm);
					\draw [fill, black] (v4) circle (0.05cm);
					\draw [fill, black] (v5) circle (0.05cm);
					\draw [fill, black] (v6) circle (0.05cm);
					\draw [fill, black] (v62) circle (0.05cm);
					\draw [fill, black] (v7) circle (0.05cm);
					\draw [fill, black] (v8) circle (0.05cm);
					\draw [fill, black] (v9) circle (0.05cm);
					\node[above right] at (v1) {$\vev^{(1)}=(19,9)$};
					\node[above right] at (v2) {$\vev^{(2)}=(27,1)$};
					\node[right] at (v5) {$\vev^{(5)}=(27\frac{27}{100}, 0)$};
					\node[below] at (v7) {$\vev^{(8)}=(27,-1)$};
					\node[left] at (v9) {$(0,0)=\vev^{(10)}$};
					\node[below ] at (v8) {$(1,-1)=\vev^{(9)}$};
					\node[above left ] at (v10) {$(9,9)=\vev^{(11)}$};
					\draw[dotted] (27,0.1) circle (1.5cm);					  
					\draw[black] (v1)--(v2)--(v3)--(v4)--(v5)--(v6)--(v62)--(v7)--(v8)--(v9)--(v10)--(v1);
					\draw[line width = 1.5, red, ->] (v1)--(v10);
					\draw[line width = 1.5, red, ->] (v10)--(v9);
					\draw[line width = 1.5, red, ->] (v9)--(v5);
			\end{tikzpicture} 
			\hspace{1.5cm}
			\begin{tikzpicture}[scale=2]
					\coordinate (v1) at (26.5,1.5);
					\coordinate (v2) at (27,1);
					\coordinate (v3) at (27.24,0.2);
					\coordinate (v3b) at (27.5,0.4);
					\coordinate (v4) at (27.26,0.1);
					\coordinate (v4b) at (27,0.2);
					\coordinate (v5) at (27.27,0);
					\coordinate (v6) at (27.26,-0.1);
					\coordinate (v6b) at (27,-0.2);
					\coordinate (v62) at (27.24,-0.2);
					\coordinate (v62b) at (27.5,-0.4);
					\coordinate (v7) at (27,-1);
					\coordinate (v8) at (26,-1);
					\draw[loosely dotted]  (27,0.1) circle (1.5cm);

					\node[above right] at (v2) {$\vev^{(2)}=(27,1)$};
					\node[right] at (v3b) {$\vev^{(3)}=(27\frac{24}{100},\frac{2}{10})$};
					\draw[line width = 0.2 ] (v3b)--(v3);
					\node[left] at (v4b) {$(27\frac{26}{100},\frac{1}{10})=\vev^{(4)}$};
					\draw[line width = 0.2 ] (v4b)--(v4);
					\node[right] at (v5) {$\vev^{(5)}=(27\frac{27}{100}, 0)$};
					\node[left] at (v6b) {$(27\frac{26}{100},-\frac{1}{10})=\vev^{(6)}$};
					\draw[line width = 0.2 ] (v6b)--(v6);
					\node[right] at (v62b) {$\vev^{(7)}=(27\frac{24}{100},-\frac{2}{10})$};
					\draw[line width = 0.2 ] (v62b)--(v62);
					\node[right] at (v7) {$\vev^{(8)}=(27,1)$};

					\draw [fill, black] (v2) circle (0.02cm);
					\draw [fill, black] (v3) circle (0.02cm);
					\draw [fill, black] (v4) circle (0.02cm);
					\draw [fill, black] (v5) circle (0.04cm);
					\draw [fill, black] (v6) circle (0.02cm);
					\draw [fill, black] (v62) circle (0.02cm);
					\draw [fill, black] (v7) circle (0.02cm);
					\draw (v1)--(v2)--(v3)--(v4)--(v5)--(v6)--(v62)--(v7)--(v8);		
					\draw[line width = 1.5, red, ->] (25.5,0)--(v5);
			\end{tikzpicture}
	\caption{A polytope with a feasible maximal backwards circuit walk of length three.}
	\label{fig: backwards walk}
	\end{figure}
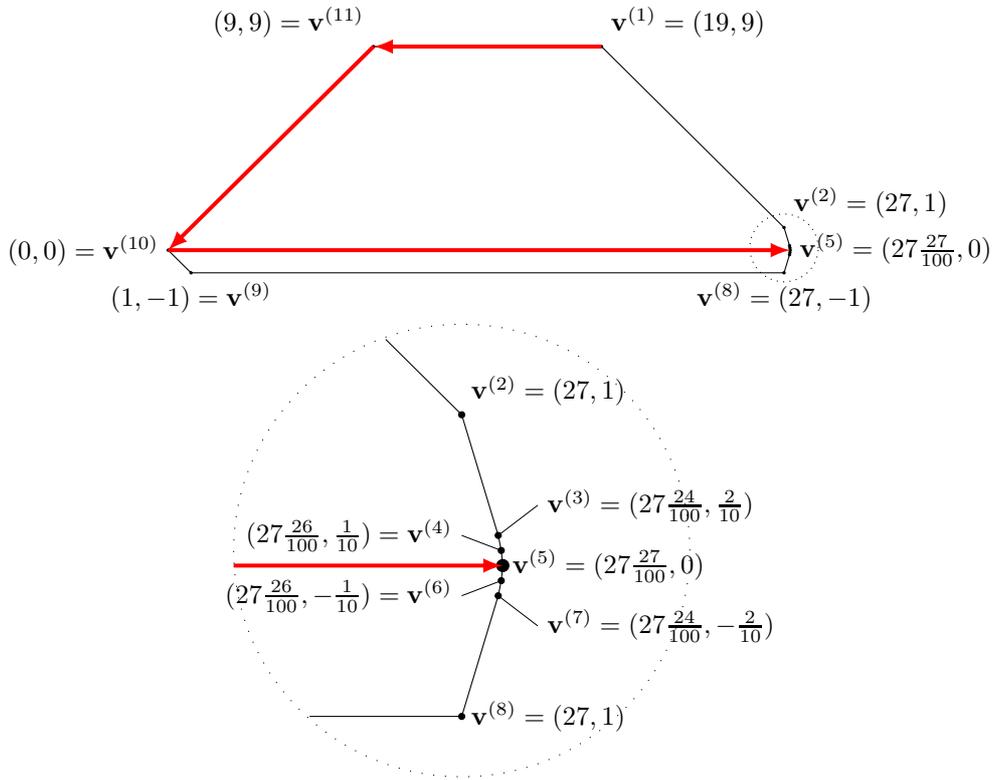	

Every other feasible maximal circuit walk from $\vev^{(1)}$ to $\vev^{(5)}$ has length at least four. 
To see this, we illustrate all possible combination of first (dashed) and second feasible maximal circuit steps in Figure \ref{Fig: first two steps CD_efmb}. From none of these second step points we can reach $\vev^{(5)}$ in only one additional step, except from the point $\vev^{(10)}$ in the top left picture. But this is the backwards circuit walk as depicted in Figure \ref{fig: backwards walk}. Observe that all second steps that end in the edge $(\vev^{(11)}, \vev^{(1)})$ have coordinates $(x,9)^T$ for an \emph{integral} $x$, in particular we cannot go to $\vev^{(5)}=(27\frac{27}{100}, 0)^T$ by applying the circuit $(-1,1)^T$ at these points. The final sketch is a zoomed-in view on the bottom right picture. It illustrates all possible second steps after applying $(-1,1)^T$ at $\vev^{(1)}$. 
\noindent	
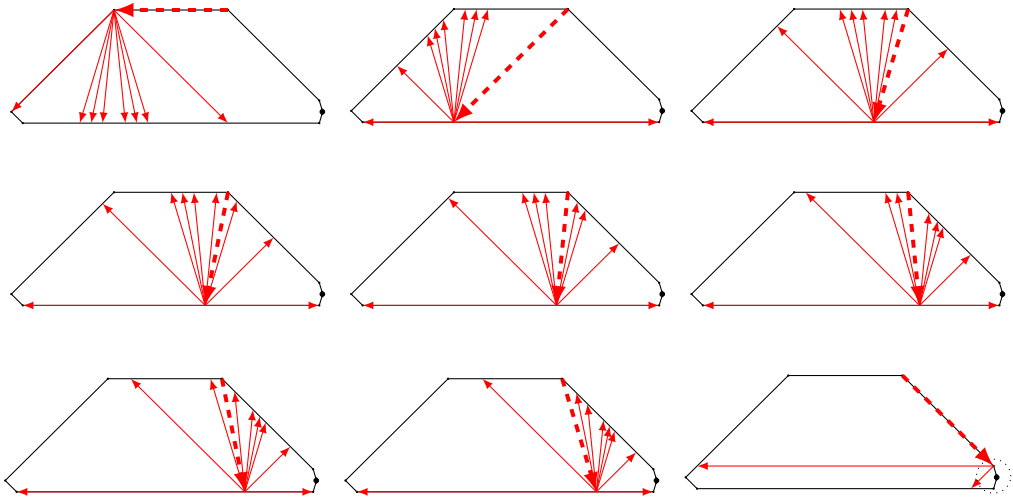
\begin{figure}[H]
		\centering
			\begin{tikzpicture}[scale=0.15]
					\coordinate (v1) at (19,9);
					\coordinate (v2) at (27,1);
					\coordinate (v3) at (27.24,0.2);
					\coordinate (v4) at (27.26,0.1);
					\coordinate (v5) at (27.27,0);
					\coordinate (v6) at (27.26,-0.1);
					\coordinate (v62) at (27.24,-0.2);
					\coordinate (v7) at (27,-1);
					\coordinate (v8) at (1,-1);
					\coordinate (v9) at (0,0);
					\coordinate (v10) at (9,9);
					
					\draw [fill, black] (v10) circle (0.05cm);
					\draw [fill, black] (v1) circle (0.05cm);
					\draw [fill, black] (v2) circle (0.05cm);
					\draw [fill, black] (v3) circle (0.05cm);
					\draw [fill, black] (v4) circle (0.05cm);
					\draw [fill, black] (v5) circle (0.2cm);
					\draw [fill, black] (v6) circle (0.05cm);
					\draw [fill, black] (v62) circle (0.05cm);
					\draw [fill, black] (v7) circle (0.05cm);
					\draw [fill, black] (v8) circle (0.05cm);
					\draw [fill, black] (v9) circle (0.05cm);
					\draw[black] (v1)--(v2)--(v3)--(v4)--(v5)--(v6)--(v62)--(v7)--(v8)--(v9)--(v10)--(v1);
					\draw[dashed, line width = 1.5, red, ->] (v1)--(v10);
					\draw[ red, ->] (v10)--(v9);
					\draw[ red, ->] (v10)--(8,-1);
					\draw[red, ->] (v10)--(7, -1);
					\draw[red, ->] (v10)--(6, -1);
					\draw[red, ->] (v10)--(10,-1);				
					\draw[red, ->] (v10)--(11,-1);				
					\draw[red, ->] (v10)--(12,-1);				
					\draw[red, ->] (v10)--(19,-1);									
			\end{tikzpicture} 
			\;
			\begin{tikzpicture}[scale=0.15]
					\coordinate (v1) at (19,9);
					\coordinate (v2) at (27,1);
					\coordinate (v3) at (27.24,0.2);
					\coordinate (v4) at (27.26,0.1);
					\coordinate (v5) at (27.27,0);
					\coordinate (v6) at (27.26,-0.1);
					\coordinate (v62) at (27.24,-0.2);
					\coordinate (v7) at (27,-1);
					\coordinate (v8) at (1,-1);
					\coordinate (v9) at (0,0);
					\coordinate (v10) at (9,9);
					
					\draw [fill, black] (v10) circle (0.05cm);
					\draw [fill, black] (v1) circle (0.05cm);
					\draw [fill, black] (v2) circle (0.05cm);
					\draw [fill, black] (v3) circle (0.05cm);
					\draw [fill, black] (v4) circle (0.05cm);
					\draw [fill, black] (v5) circle (0.2cm);
					\draw [fill, black] (v6) circle (0.05cm);
					\draw [fill, black] (v62) circle (0.05cm);
					\draw [fill, black] (v7) circle (0.05cm);
					\draw [fill, black] (v8) circle (0.05cm);
					\draw [fill, black] (v9) circle (0.05cm);
					\draw[black] (v1)--(v2)--(v3)--(v4)--(v5)--(v6)--(v62)--(v7)--(v8)--(v9)--(v10)--(v1);
					\draw[dashed, line width = 1.5, red, ->] (v1)--(9,-1);
					\draw[red, ->] (9,-1)--(v7);
					\draw[red, ->] (9,-1)--(v8);
					\draw[red, ->] (9,-1)--(10,9);
					\draw[red, ->] (9,-1)--(11,9);
					\draw[red, ->] (9,-1)--(12,9);
					\draw[red, ->] (9,-1)--(8.09,8.09); 
					\draw[red, ->] (9,-1)--(7.33,7.33); 
					\draw[red, ->] (9,-1)--(6.69, 6.69); 
					\draw[red, ->] (9,-1)--(4,4);
			\end{tikzpicture} 
			\;
			\begin{tikzpicture}[scale=0.15]
					\coordinate (v1) at (19,9);
					\coordinate (v2) at (27,1);
					\coordinate (v3) at (27.24,0.2);
					\coordinate (v4) at (27.26,0.1);
					\coordinate (v5) at (27.27,0);
					\coordinate (v6) at (27.26,-0.1);
					\coordinate (v62) at (27.24,-0.2);
					\coordinate (v7) at (27,-1);
					\coordinate (v8) at (1,-1);
					\coordinate (v9) at (0,0);
					\coordinate (v10) at (9,9);
					
					\draw [fill, black] (v10) circle (0.05cm);
					\draw [fill, black] (v1) circle (0.05cm);
					\draw [fill, black] (v2) circle (0.05cm);
					\draw [fill, black] (v3) circle (0.05cm);
					\draw [fill, black] (v4) circle (0.05cm);
					\draw [fill, black] (v5) circle (0.2cm);
					\draw [fill, black] (v6) circle (0.05cm);
					\draw [fill, black] (v62) circle (0.05cm);
					\draw [fill, black] (v7) circle (0.05cm);
					\draw [fill, black] (v8) circle (0.05cm);
					\draw [fill, black] (v9) circle (0.05cm);
					\draw[black] (v1)--(v2)--(v3)--(v4)--(v5)--(v6)--(v62)--(v7)--(v8)--(v9)--(v10)--(v1);
					\draw[dashed, line width = 1.5, red, ->] (v1)--(16,-1);
					\draw[red, ->] (16,-1)--(17,9);				
					\draw[red, ->] (16,-1)--(18,9);				
					\draw[red, ->] (16,-1)--(15,9);				
					\draw[red, ->] (16,-1)--(14,9);				
					\draw[red, ->] (16,-1)--(13,9);				
					\draw[red, ->] (16,-1)--(7.5,7.5);				
					\draw[red, ->] (16,-1)--(v8);				
					\draw[red, ->] (16,-1)--(v7);				
					\draw[red, ->] (16,-1)--(22.5,5.5);				
			\end{tikzpicture} 
	\end{figure}
	\begin{figure}[H]
			\centering
		\begin{tikzpicture}[scale=0.15]
					\coordinate (v1) at (19,9);
					\coordinate (v2) at (27,1);
					\coordinate (v3) at (27.24,0.2);
					\coordinate (v4) at (27.26,0.1);
					\coordinate (v5) at (27.27,0);
					\coordinate (v6) at (27.26,-0.1);
					\coordinate (v62) at (27.24,-0.2);
					\coordinate (v7) at (27,-1);
					\coordinate (v8) at (1,-1);
					\coordinate (v9) at (0,0);
					\coordinate (v10) at (9,9);
					
					\draw [fill, black] (v10) circle (0.05cm);
					\draw [fill, black] (v1) circle (0.05cm);
					\draw [fill, black] (v2) circle (0.05cm);
					\draw [fill, black] (v3) circle (0.05cm);
					\draw [fill, black] (v4) circle (0.05cm);
					\draw [fill, black] (v5) circle (0.2cm);
					\draw [fill, black] (v6) circle (0.05cm);
					\draw [fill, black] (v62) circle (0.05cm);
					\draw [fill, black] (v7) circle (0.05cm);
					\draw [fill, black] (v8) circle (0.05cm);
					\draw [fill, black] (v9) circle (0.05cm);
					\draw[black] (v1)--(v2)--(v3)--(v4)--(v5)--(v6)--(v62)--(v7)--(v8)--(v9)--(v10)--(v1);
					\draw[dashed, line width = 1.5, red, ->] (v1)--(17,-1);
					\draw[red, ->] (17,-1)--(16,9);				
					\draw[red, ->] (17,-1)--(15,9);				
					\draw[red, ->] (17,-1)--(14,9);				
					\draw[red, ->] (17,-1)--(18,9);				
					\draw[red, ->] (17,-1)--(8,8);				
					\draw[red, ->] (17,-1)--(23,5);				
					\draw[red, ->] (17,-1)--(19.76923077,8.230769231);				
					\draw[red, ->] (17,-1)--(v8);				
					\draw[red, ->] (17,-1)--(v7);				
			\end{tikzpicture} 
			\;
			\begin{tikzpicture}[scale=0.15]
					\coordinate (v1) at (19,9);
					\coordinate (v2) at (27,1);
					\coordinate (v3) at (27.24,0.2);
					\coordinate (v4) at (27.26,0.1);
					\coordinate (v5) at (27.27,0);
					\coordinate (v6) at (27.26,-0.1);
					\coordinate (v62) at (27.24,-0.2);
					\coordinate (v7) at (27,-1);
					\coordinate (v8) at (1,-1);
					\coordinate (v9) at (0,0);
					\coordinate (v10) at (9,9);
					
					\draw [fill, black] (v10) circle (0.05cm);
					\draw [fill, black] (v1) circle (0.05cm);
					\draw [fill, black] (v2) circle (0.05cm);
					\draw [fill, black] (v3) circle (0.05cm);
					\draw [fill, black] (v4) circle (0.05cm);
					\draw [fill, black] (v5) circle (0.2cm);
					\draw [fill, black] (v6) circle (0.05cm);
					\draw [fill, black] (v62) circle (0.05cm);
					\draw [fill, black] (v7) circle (0.05cm);
					\draw [fill, black] (v8) circle (0.05cm);
					\draw [fill, black] (v9) circle (0.05cm);
					\draw[black] (v1)--(v2)--(v3)--(v4)--(v5)--(v6)--(v62)--(v7)--(v8)--(v9)--(v10)--(v1);
					\draw[dashed, line width = 1.5, red, ->] (v1)--(18,-1);
					\draw[red, ->] (18,-1)--(17,9);				
					\draw[red, ->] (18,-1)--(16,9);				
					\draw[red, ->] (18,-1)--(15,9);				
					\draw[red, ->] (18,-1)--(19.8333,8.1667);				
					\draw[red, ->] (18,-1)--(20.53846154,7.461538462);				

					\draw[red, ->] (18,-1)--(v7);				
					\draw[red, ->] (18,-1)--(v8);				
					\draw[red, ->] (18,-1)--(8.5,8.5);				
					\draw[red, ->] (18,-1)--(23.5,4.5);				
			\end{tikzpicture} 
			\;
			\begin{tikzpicture}[scale=0.15]
					\coordinate (v1) at (19,9);
					\coordinate (v2) at (27,1);
					\coordinate (v3) at (27.24,0.2);
					\coordinate (v4) at (27.26,0.1);
					\coordinate (v5) at (27.27,0);
					\coordinate (v6) at (27.26,-0.1);
					\coordinate (v62) at (27.24,-0.2);
					\coordinate (v7) at (27,-1);
					\coordinate (v8) at (1,-1);
					\coordinate (v9) at (0,0);
					\coordinate (v10) at (9,9);
					
					\draw [fill, black] (v10) circle (0.05cm);
					\draw [fill, black] (v1) circle (0.05cm);
					\draw [fill, black] (v2) circle (0.05cm);
					\draw [fill, black] (v3) circle (0.05cm);
					\draw [fill, black] (v4) circle (0.05cm);
					\draw [fill, black] (v5) circle (0.2cm);
					\draw [fill, black] (v6) circle (0.05cm);
					\draw [fill, black] (v62) circle (0.05cm);
					\draw [fill, black] (v7) circle (0.05cm);
					\draw [fill, black] (v8) circle (0.05cm);
					\draw [fill, black] (v9) circle (0.05cm);
					\draw[black] (v1)--(v2)--(v3)--(v4)--(v5)--(v6)--(v62)--(v7)--(v8)--(v9)--(v10)--(v1);
					\draw[dashed, line width = 1.5, red, ->] (v1)--(20,-1);
					\draw[red, ->] (20,-1)--(18,9);				
					\draw[red, ->] (20,-1)--(17,9);				
					\draw[red, ->] (20,-1)--(20.81818,7.18181);				
					\draw[red, ->] (20,-1)--(21.6,6.5);				
					\draw[red, ->] (20,-1)--(22.07692308,5.923076923);				
					\draw[red, ->] (20,-1)--(24.5,3.5);				
					\draw[red, ->] (20,-1)--(10,9);				
					\draw[red, ->] (20,-1)--(v8);				
					\draw[red, ->] (20,-1)--(v7);				
			\end{tikzpicture} 
	\end{figure}
	\begin{figure}[H]
		\centering
			\begin{tikzpicture}[scale=0.15]
					\coordinate (v1) at (19,9);
					\coordinate (v2) at (27,1);
					\coordinate (v3) at (27.24,0.2);
					\coordinate (v4) at (27.26,0.1);
					\coordinate (v5) at (27.27,0);
					\coordinate (v6) at (27.26,-0.1);
					\coordinate (v62) at (27.24,-0.2);
					\coordinate (v7) at (27,-1);
					\coordinate (v8) at (1,-1);
					\coordinate (v9) at (0,0);
					\coordinate (v10) at (9,9);
					
					\draw [fill, black] (v10) circle (0.05cm);
					\draw [fill, black] (v1) circle (0.05cm);
					\draw [fill, black] (v2) circle (0.05cm);
					\draw [fill, black] (v3) circle (0.05cm);
					\draw [fill, black] (v4) circle (0.05cm);
					\draw [fill, black] (v5) circle (0.2cm);
					\draw [fill, black] (v6) circle (0.05cm);
					\draw [fill, black] (v62) circle (0.05cm);
					\draw [fill, black] (v7) circle (0.05cm);
					\draw [fill, black] (v8) circle (0.05cm);
					\draw [fill, black] (v9) circle (0.05cm);
					\draw[black] (v1)--(v2)--(v3)--(v4)--(v5)--(v6)--(v62)--(v7)--(v8)--(v9)--(v10)--(v1);
					\draw[dashed, line width = 1.5, red, ->] (v1)--(21,-1);
					\draw[red, ->] (21,-1)--(18,9);				
					\draw[red, ->] (21,-1)--(20.111,7.889);				
					\draw[red, ->] (21,-1)--(22.84615385, 5.153846154);				
					\draw[red, ->] (21,-1)--(21.7272, 6.2727);				
					\draw[red, ->] (21,-1)--(22.333,5.667);				
					\draw[red, ->] (21,-1)--(v7);				
					\draw[red, ->] (21,-1)--(v8);				
					\draw[red, ->] (21,-1)--(25,3);				
					\draw[red, ->] (21,-1)--(11,9);				
			\end{tikzpicture} 
			\;
			\begin{tikzpicture}[scale=0.15]
					\coordinate (v1) at (19,9);
					\coordinate (v2) at (27,1);
					\coordinate (v3) at (27.24,0.2);
					\coordinate (v4) at (27.26,0.1);
					\coordinate (v5) at (27.27,0);
					\coordinate (v6) at (27.26,-0.1);
					\coordinate (v62) at (27.24,-0.2);
					\coordinate (v7) at (27,-1);
					\coordinate (v8) at (1,-1);
					\coordinate (v9) at (0,0);
					\coordinate (v10) at (9,9);
					
					\draw [fill, black] (v10) circle (0.05cm);
					\draw [fill, black] (v1) circle (0.05cm);
					\draw [fill, black] (v2) circle (0.05cm);
					\draw [fill, black] (v3) circle (0.05cm);
					\draw [fill, black] (v4) circle (0.05cm);
					\draw [fill, black] (v5) circle (0.2cm);
					\draw [fill, black] (v6) circle (0.05cm);
					\draw [fill, black] (v62) circle (0.05cm);
					\draw [fill, black] (v7) circle (0.05cm);
					\draw [fill, black] (v8) circle (0.05cm);
					\draw [fill, black] (v9) circle (0.05cm);
					\draw[black] (v1)--(v2)--(v3)--(v4)--(v5)--(v6)--(v62)--(v7)--(v8)--(v9)--(v10)--(v1);
					\draw[dashed, line width = 1.5, red, ->] (v1)--(22,-1);
					\draw[red, ->] (22,-1)--(v7);				
					\draw[red, ->] (22,-1)--(v8);				
					\draw[red, ->] (22,-1)--(25.5,2.5);				
					\draw[red, ->] (22,-1)--(12,9);				
					\draw[red, ->] (22,-1)--(21.222,6.778);				
					\draw[red, ->] (22,-1)--(20.25,7.75);				
					\draw[red, ->] (22,-1)--(22.6363,5.3636);				
					\draw[red, ->] (22,-1)--(23.1667,4.8333);				
					\draw[red, ->] (22,-1)--(23.61538462, 4.384615385);				
			\end{tikzpicture} 
			\;
			\begin{tikzpicture}[scale=0.15]
					\coordinate (v1) at (19,9);
					\coordinate (v2) at (27,1);
					\coordinate (v3) at (27.24,0.2);
					\coordinate (v4) at (27.26,0.1);
					\coordinate (v5) at (27.27,0);
					\coordinate (v6) at (27.26,-0.1);
					\coordinate (v62) at (27.24,-0.2);
					\coordinate (v7) at (27,-1);
					\coordinate (v8) at (1,-1);
					\coordinate (v9) at (0,0);
					\coordinate (v10) at (9,9);
					
					\draw [fill, black] (v10) circle (0.05cm);
					\draw [fill, black] (v1) circle (0.05cm);
					\draw [fill, black] (v2) circle (0.05cm);
					\draw [fill, black] (v3) circle (0.05cm);
					\draw [fill, black] (v4) circle (0.05cm);
					\draw [fill, black] (v5) circle (0.2cm);
					\draw [fill, black] (v6) circle (0.05cm);
					\draw [fill, black] (v62) circle (0.05cm);
					\draw [fill, black] (v7) circle (0.05cm);
					\draw [fill, black] (v8) circle (0.05cm);
					\draw [fill, black] (v9) circle (0.05cm);
					\draw[black] (v1)--(v2)--(v3)--(v4)--(v5)--(v6)--(v62)--(v7)--(v8)--(v9)--(v10)--(v1);
					\draw[dashed, line width = 1.5, red, ->] (v1)--(v2); 
					\draw[red, ->] (v2)--(1,1);				
					\draw[red, ->] (v2)--(25,-1);				
					\draw[dotted] (27,0.1) circle (1.5cm);					  
			\end{tikzpicture} 
	\end{figure}
	\begin{figure}[H]
		\centering
			\begin{tikzpicture}[scale=1.5]
					\coordinate (v1) at (26.5,1.5);
					\coordinate (v2) at (27,1);
					\coordinate (v3) at (27.24,0.2);
					\coordinate (v3b) at (27.5,0.4);
					\coordinate (v4) at (27.26,0.1);
					\coordinate (v4b) at (27,0.2);
					\coordinate (v5) at (27.27,0);
					\coordinate (v6) at (27.26,-0.1);
					\coordinate (v6b) at (27,-0.2);
					\coordinate (v62) at (27.24,-0.2);
					\coordinate (v62b) at (27.5,-0.4);
					\coordinate (v7) at (27,-1);
					\coordinate (v8) at (26,-1);
					\draw[loosely dotted]  (27,0.1) circle (1.5cm);

					\node[right] at (v5) {$\vev^{(5)}$};

					\draw [fill, black] (v2) circle (0.02cm);
					\draw [fill, black] (v3) circle (0.02cm);
					\draw [fill, black] (v4) circle (0.02cm);
					\draw [fill, black] (v5) circle (0.04cm);
					\draw [fill, black] (v6) circle (0.02cm);
					\draw [fill, black] (v62) circle (0.02cm);
					\draw [fill, black] (v7) circle (0.02cm);
					\draw (v1)--(v2)--(v3)--(v4)--(v5)--(v6)--(v62)--(v7)--(v8);		
					\draw[dashed, line width = 1.5, red, ->] (v1)--(v2);
					\draw[red, ->] (v2)--(v3);				
					\draw[red, ->] (v2)--(26.8,-1);
					\draw[red, ->] (v2)--(26.6,-1);
					\draw[red, ->] (v2)--(26.4,-1);
					\draw[red, ->] (v2)--(27.15,-0.5);
					\draw[red, ->] (v2)--(27.24,-0.2);
					\draw[line width = 0.1] (27.24,0.2)-- (27.8,0.2);
					\draw[line width = 0.1] (27.24,-0.2)-- (27.8,-0.2);
					\draw[line width = 0.1] (27.15,-0.5)-- (27.5,-0.5);
					\node[right] at (27.8,0.2) {\footnotesize$\vev^{(3)}$};
					\node[right] at (27.8,-0.2) {\footnotesize$\vev^{(7)}$};
					\node[right] at (27.5,-0.5) {\footnotesize$(27\frac {15}{100},-\frac 12)$};	
			\end{tikzpicture}
		\caption{Possible combinations of first and second feasible maximal circuit steps from $\vev^{(1)}$.}
		\label{Fig: first two steps CD_efmb}
		\end{figure}
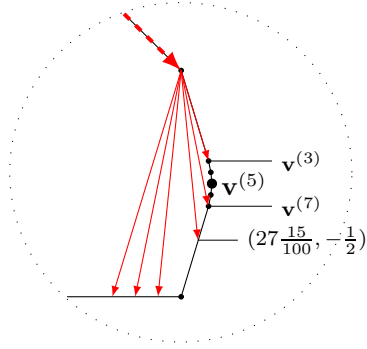
\end{proof}

Similarly, one may obtain a shorter circuit walk by to allowing oneself to use a repeated circuit.

\begin{lemma}[$\CD_{fmr}>\CD_{fm}$]\label{ex: repeated is stronger}
For $n=2$, there is a polytope with a pair of vertices for which every optimal feasible maximal circuit walk is repetitive. Hence the distances $\CD_{fmr}$ and $\CD_{fm}$ differ in this case.
\end{lemma}
\begin{proof}
We consider the following polytope on nine vertices depicted in Figure \ref{Fig:repetitiveneeded}. Note that there are two edges $e_0$ and $e_7$ with direction $(1,0)^T$, an edge $e_1$ with direction $(1,-1)^T$, an edge $e_6$ with direction $(1,1)^T$ and the edge $e_8$ with direction $(0,1)^T$. Further, in the right part there are four steeper edges: $e_2$ with direction $(1,-4)^T$, $e_3$ with direction $(1,-5)^T$, $e_4$ with direction $(1,5)^T$, $e_5$ with direction $(1,4)^T$.
	\begin{figure}[H]
		\centering
			\begin{tikzpicture}[scale=0.24]
					\coordinate (v0) at (0,10);
					\coordinate (v1) at (10,10);
					\coordinate (v2) at (19,1);
					\coordinate (v3) at (19.125,0.5);
					\coordinate (v4) at (19.225,0);
					\coordinate (v5) at (19.125,-0.5);
					\coordinate (v6) at (19,-1);
					\coordinate (v7) at (10,-10);
					\coordinate (v8) at (0,-10);
					
					\draw [fill, black] (v0) circle (0.05cm);
					\draw [fill, black] (v1) circle (0.05cm);
					\draw [fill, black] (v2) circle (0.05cm);
					\draw [fill, black] (v3) circle (0.05cm);
					\draw [fill, black] (v4) circle (0.05cm);
					\draw [fill, black] (v5) circle (0.05cm);
					\draw [fill, black] (v6) circle (0.05cm);
					\draw [fill, black] (v7) circle (0.05cm);
					\draw [fill, black] (v8) circle (0.05cm);
					\node[above right] at (v1) {$\vev^{(2)}=(10,10)$};
					\node[above right] at (v2) {$\vev^{(3)}$};
					\node[right] at (v4) {$\vev^{(5)}$};
					\node[below right] at (v6) {$\vev^{(7)}$};
					\node[below right] at (v7) {$\vev^{(8)}=(10,-10)$};
					\node[above] at (v0) {$(0,10)=\vev^{(1)}$};
					\node[below ] at (v8) {$(0,-10)=\vev^{(9)}$};
					\draw[dotted] (v4) circle (1.8cm);					  
\draw[black] (v0) node[ right] {$e_0$}--(v1) node[below right] {$e_1$}--(v2)--(v3)--(v4)--(v5)--(v6)--(v7)node[above right] {$e_6$}--(v8)node[  right]{$e_7$}--(v0);	
					\node[left] at (0,-5) {$e_8$};
			\end{tikzpicture} 
			\hspace{1.5cm}
			\begin{tikzpicture}[scale=1.5]
					\coordinate (v1) at (18.4,1.6);
					\coordinate (v2) at (19,1);
					\coordinate (v3) at (19.125,0.5);
					\coordinate (v4) at (19.225,0);
					\coordinate (v5) at (19.125,-0.5);
					\coordinate (v6) at (19,-1);
					\coordinate (v7) at (18.4,-1.6);
					\draw[loosely dotted] (v4) circle (1.8cm);
					\node[above right] at (v2) {$\vev^{(3)}=(19,1)$};
					\node[right] at (v3) {$\vev^{(4)}=(19\frac 18,\frac 12)$};
					\node[right] at (v4) {$\vev^{(5)}=(19\frac {9}{40},0)$};
					\node[right] at (v5) {$\vev^{(6)}=(19\frac 18,-\frac 12)$};
					\node[below right] at (v6) {$\vev^{(7)}=(19,-1)$};
					\draw [fill, black] (v2) circle (0.01cm);
					\draw [fill, black] (v3) circle (0.01cm);
					\draw [fill, black] (v4) circle (0.01cm);
					\draw [fill, black] (v5) circle (0.01cm);
					\draw [fill, black] (v6) circle (0.01cm);
					\draw (v1)--(v2)node[below left] {$e_2$}--(v3)node[below left] {$e_3$}--(v4)--(v5)node[above left] {$e_4$}--(v6)node[above left] {$e_5$}--(v7);		
			\end{tikzpicture}
	\caption{The polytope for the proof of Lemma \ref{ex: repeated is stronger}.}\label{Fig:repetitiveneeded}
	\end{figure}
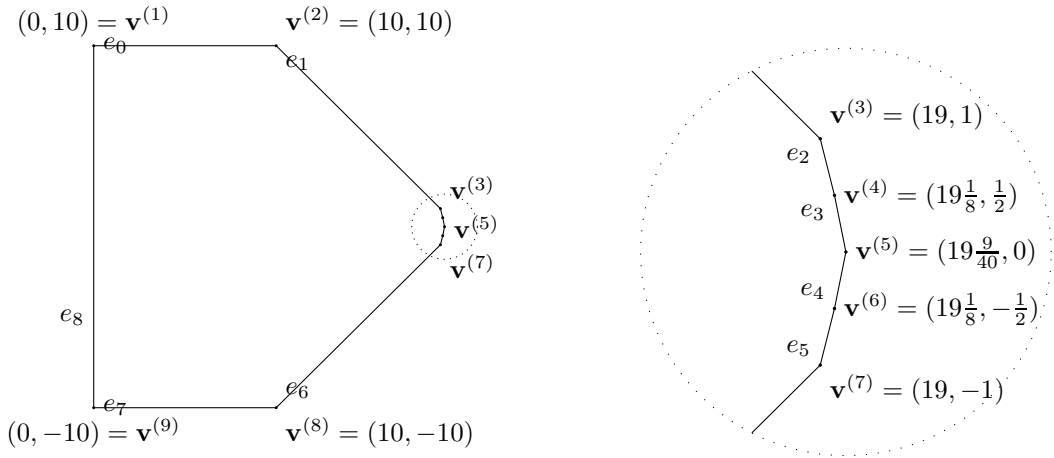	

There is a feasible maximal circuit walk of length three from $\vev^{(1)}$ to $\vev^{(5)}$ that is repetitive.
	\begin{figure}[H]
		\centering
			\begin{tikzpicture}[scale=0.2]
					\coordinate (v0) at (0,10);
					\coordinate (v1) at (10,10);
					\coordinate (v2) at (19,1);
					\coordinate (v3) at (19.125,0.5);
					\coordinate (v4) at (19.225,0);
					\coordinate (v5) at (19.125,-0.5);
					\coordinate (v6) at (19,-1);
					\coordinate (v7) at (10,-10);
					\coordinate (v8) at (0,-10);
					
					\draw [fill, black] (v0) circle (0.05cm);
					\draw [fill, black] (v1) circle (0.05cm);
					\draw [fill, black] (v2) circle (0.05cm);
					\draw [fill, black] (v3) circle (0.05cm);
					\draw [fill, black] (v4) circle (0.05cm);
					\draw [fill, black] (v5) circle (0.05cm);
					\draw [fill, black] (v6) circle (0.05cm);
					\draw [fill, black] (v7) circle (0.05cm);
					\draw [fill, black] (v8) circle (0.05cm);
					\node[above] at (v1) {$\vev^{(2)}=(10,10)$};
					\node[right] at (v4) {$\vev^{(5)}=(19\frac{9}{40},0)$};
					\node[above left] at (v0) {$(0,10)=\vev^{(1)}$};
					\node[left] at (0,0) {$(0,0)$};
					\draw[black] (v0) --(v1)--(v2)--(v3)--(v4)--(v5)--(v6)--(v7)--(v8)--(v0);
					\draw[ line width= 1.5, red, ->] (v0)--(v1);
					\draw[line width= 1.5, red, ->] (v1)--(0,0);
					\draw[line width= 1.5, red, ->] (0,0)--(v4);				
			\end{tikzpicture} 
	\caption{A feasible maximal repetitive circuit walk of length three.}
	\label{fig: repetitive walk}
	\end{figure}
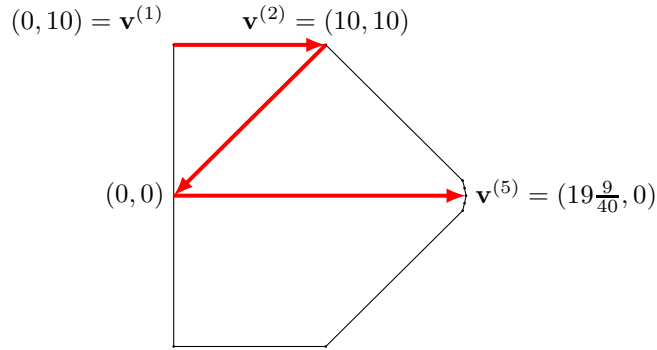	

Every other circuit walk from $\vev^{(1)}$ to $\vev^{(5)}$ has length at least four. 
Therefore we illustrate all possible combination of first (dashed) and second steps in Figure \ref{Fig: first two steps CD_efmr}. From none of these second step points we can reach $\vev^{(5)}$ in only one additional step, except from the point $(0,0)$ in the first picture. But this is the repetitive circuit walk as in Figure \ref{fig: repetitive walk}. For those points for which it might not be immediately obvious that we cannot get to $\vev^{(5)}$ in only one more step, we added the coordinates for a convenient verification that we cannot reach $\vev^{(5)}=(19\frac{9}{40},0)$ with any edge direction.   
	\begin{figure}[H]
		\centering
			\begin{tikzpicture}[scale=0.18]
					\coordinate (v0) at (0,10);
					\coordinate (v1) at (10,10);
					\coordinate (v2) at (19,1);
					\coordinate (v3) at (19.125,0.5);
					\coordinate (v4) at (19.225,0);
					\coordinate (v5) at (19.125,-0.5);
					\coordinate (v6) at (19,-1);
					\coordinate (v7) at (10,-10);
					\coordinate (v8) at (0,-10);					
					\draw [fill, black] (v0) circle (0.05cm);
					\draw [fill, black] (v1) circle (0.05cm);
					\draw [fill, black] (v2) circle (0.05cm);
					\draw [fill, black] (v3) circle (0.05cm);
					\draw [fill, black] (v4) circle (0.2cm);
					\draw [fill, black] (v5) circle (0.05cm);
					\draw [fill, black] (v6) circle (0.05cm);
					\draw [fill, black] (v7) circle (0.05cm);
					\draw [fill, black] (v8) circle (0.05cm);
					\draw [fill, black] (0,0) circle (0.2cm);
					\draw[black] (v0) --(v1)--(v2)--(v3)--(v4)--(v5)--(v6)--(v7)--(v8)--(v0);
					\draw[dashed, line width= 1.5, red, ->] (v0)--(v1);
					\draw[line width= 1.5, red, ->] (v1)--(0,0);
					\draw[line width= 1.5, red, ->] (v1)--(6,-10);
					\draw[line width= 1.5, red, ->] (v1)--(5,-10);
					\draw[line width= 1.5, red, ->] (v1)--(v7);
					\draw[line width= 1.5, red, ->] (v1)--(v2);
					\draw[line width= 1.5, red, ->] (v1)--(14,-6);
					\draw[line width= 1.5, red, ->] (v1)--(13.33,-6.67);					
			\end{tikzpicture} 
			\qquad
			\begin{tikzpicture}[scale=0.18]
					\coordinate (v0) at (0,10);
					\coordinate (v1) at (10,10);
					\coordinate (v2) at (19,1);
					\coordinate (v3) at (19.125,0.5);
					\coordinate (v4) at (19.225,0);
					\coordinate (v5) at (19.125,-0.5);
					\coordinate (v6) at (19,-1);
					\coordinate (v7) at (10,-10);
					\coordinate (v8) at (0,-10);					
					\draw [fill, black] (v0) circle (0.05cm);
					\draw [fill, black] (v1) circle (0.05cm);
					\draw [fill, black] (v2) circle (0.05cm);
					\draw [fill, black] (v3) circle (0.05cm);
					\draw [fill, black] (v4) circle (0.2cm);
					\draw [fill, black] (v5) circle (0.05cm);
					\draw [fill, black] (v6) circle (0.05cm);
					\draw [fill, black] (v7) circle (0.05cm);
					\draw [fill, black] (v8) circle (0.05cm);
					\draw[black] (v0) --(v1)--(v2)--(v3)--(v4)--(v5)--(v6)--(v7)--(v8)--(v0);
					\draw[dashed, line width= 1.5, red, ->] (v0)--(15,-5);
					\draw[line width= 1.5, red, ->] (15,-5)--(15,5);
					\draw[line width= 1.5, red, ->] (15,-5)--(17,3);
					\draw[line width= 1.5, red, ->] (15,-5)--(16.67,3.33);
					\draw[line width= 1.5, red, ->] (15,-5)--(v6);
					\draw[line width= 1.5, red, ->] (15,-5)--(v7);
					\draw[line width= 1.5, red, ->] (15,-5)--(12.5,7.5);
					\draw[line width= 1.5, red, ->] (15,-5)--(11.67,8.33);
					\draw[line width= 1.5, red, ->] (15,-5)--(0,-5);
					\node[above] at (16.67,3.33) {\qquad \footnotesize $(16\frac 23,3\frac 13)$};
					\draw [fill, black] (16.67,3.33) circle (0.2cm);
					\node[right] at (17,3) {\footnotesize $(17,3)$};
					\draw [fill, black] (17,3) circle (0.2cm);
			\end{tikzpicture} 
			\qquad
			\begin{tikzpicture}[scale=0.18]
					\coordinate (v0) at (0,10);
					\coordinate (v1) at (10,10);
					\coordinate (v2) at (19,1);
					\coordinate (v3) at (19.125,0.5);
					\coordinate (v4) at (19.225,0);
					\coordinate (v5) at (19.125,-0.5);
					\coordinate (v6) at (19,-1);
					\coordinate (v7) at (10,-10);
					\coordinate (v8) at (0,-10);					
					\draw [fill, black] (v0) circle (0.05cm);
					\draw [fill, black] (v1) circle (0.05cm);
					\draw [fill, black] (v2) circle (0.05cm);
					\draw [fill, black] (v3) circle (0.05cm);
					\draw [fill, black] (v4) circle (0.2cm);
					\draw [fill, black] (v5) circle (0.05cm);
					\draw [fill, black] (v6) circle (0.05cm);
					\draw [fill, black] (v7) circle (0.05cm);
					\draw [fill, black] (v8) circle (0.05cm);
					\draw[black] (v0) --(v1)--(v2)--(v3)--(v4)--(v5)--(v6)--(v7)--(v8)--(v0);
					\draw[dashed, line width= 1.5, red, ->] (v0)--(5,-10);
					\draw[line width= 1.5, red, ->] (5,-10)--(v8);
					\draw[line width= 1.5, red, ->] (5,-10)--(v7);
					\draw[line width= 1.5, red, ->] (5,-10)--(0,-5);
					\draw[line width= 1.5, red, ->] (5,-10)--(10,10);
					\draw[line width= 1.5, red, ->] (5,-10)--(9,10);
					\draw[line width= 1.5, red, ->] (5,-10)--(5,10);
					\draw[line width= 1.5, red, ->] (5,-10)--(17.5,2.5);
					\draw[line width= 1.5, red, ->] (5,-10)--(1,10);
					\node[above] at (9,10) {\footnotesize $(9,10)$};
					\draw [fill, black] (9,10) circle (0.2cm);
			\end{tikzpicture} 
	\end{figure}
	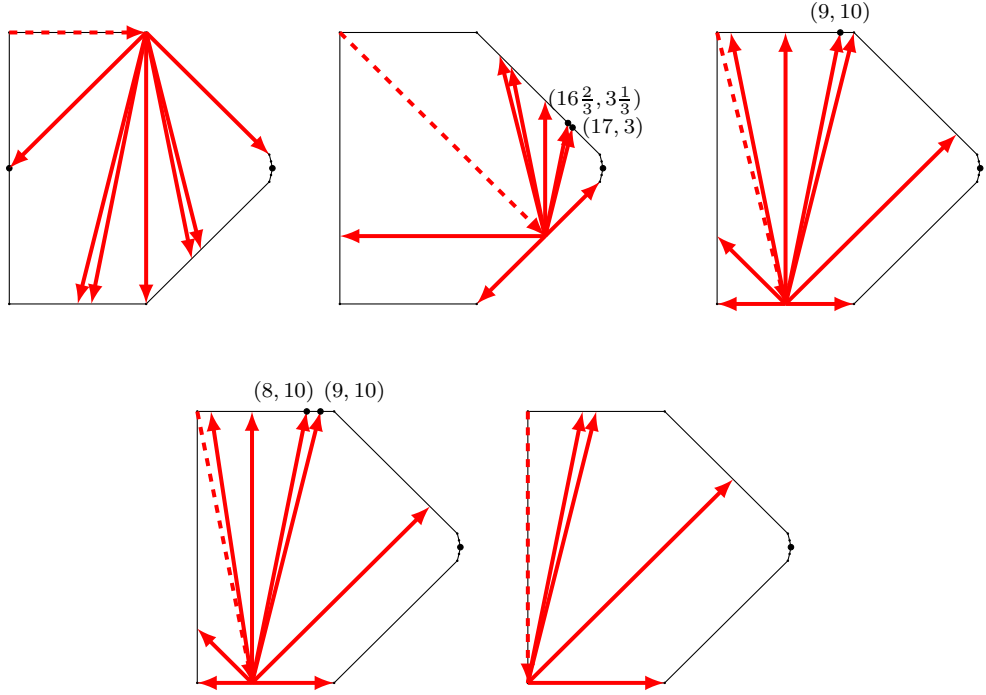
\begin{figure}[H]
		\centering
			\begin{tikzpicture}[scale=0.18]
					\coordinate (v0) at (0,10);
					\coordinate (v1) at (10,10);
					\coordinate (v2) at (19,1);
					\coordinate (v3) at (19.125,0.5);
					\coordinate (v4) at (19.225,0);
					\coordinate (v5) at (19.125,-0.5);
					\coordinate (v6) at (19,-1);
					\coordinate (v7) at (10,-10);
					\coordinate (v8) at (0,-10);					
					\draw [fill, black] (v0) circle (0.05cm);
					\draw [fill, black] (v1) circle (0.05cm);
					\draw [fill, black] (v2) circle (0.05cm);
					\draw [fill, black] (v3) circle (0.05cm);
					\draw [fill, black] (v4) circle (0.2cm);
					\draw [fill, black] (v5) circle (0.05cm);
					\draw [fill, black] (v6) circle (0.05cm);
					\draw [fill, black] (v7) circle (0.05cm);
					\draw [fill, black] (v8) circle (0.05cm);
					\draw[black] (v0) --(v1)--(v2)--(v3)--(v4)--(v5)--(v6)--(v7)--(v8)--(v0);
					\draw[dashed, line width= 1.5, red, ->] (v0)--(4,-10);
					\draw[line width= 1.5, red, ->] (4,-10)--(v8);
					\draw[line width= 1.5, red, ->] (4,-10)--(v7);
					\draw[line width= 1.5, red, ->] (4,-10)--(0,-6);
					\draw[line width= 1.5, red, ->] (4,-10)--(17,3);
					\draw[line width= 1.5, red, ->] (4,-10)--(4,10);
					\draw[line width= 1.5, red, ->] (4,-10)--(8,10);
					\draw[line width= 1.5, red, ->] (4,-10)--(9,10);
					\draw[line width= 1.5, red, ->] (4,-10)--(1,10);
					\node[above] at (9,10) {\footnotesize \quad\quad\quad$(9,10)$};
					\draw [fill, black] (9,10) circle (0.2cm);
					\node[above] at (8,10) {\footnotesize $(8,10)\quad\quad$};
					\draw [fill, black] (8,10) circle (0.2cm);
			\end{tikzpicture} 
			\qquad
			\begin{tikzpicture}[scale=0.18]
					\coordinate (v0) at (0,10);
					\coordinate (v1) at (10,10);
					\coordinate (v2) at (19,1);
					\coordinate (v3) at (19.125,0.5);
					\coordinate (v4) at (19.225,0);
					\coordinate (v5) at (19.125,-0.5);
					\coordinate (v6) at (19,-1);
					\coordinate (v7) at (10,-10);
					\coordinate (v8) at (0,-10);					
					\draw [fill, black] (v0) circle (0.05cm);
					\draw [fill, black] (v1) circle (0.05cm);
					\draw [fill, black] (v2) circle (0.05cm);
					\draw [fill, black] (v3) circle (0.05cm);
					\draw [fill, black] (v4) circle (0.2cm);
					\draw [fill, black] (v5) circle (0.05cm);
					\draw [fill, black] (v6) circle (0.05cm);
					\draw [fill, black] (v7) circle (0.05cm);
					\draw [fill, black] (v8) circle (0.05cm);
					\draw[black] (v0) --(v1)--(v2)--(v3)--(v4)--(v5)--(v6)--(v7)--(v8)--(v0);
					\draw[dashed, line width= 1.5, red, ->] (v0)--(0,-10);
					\draw[line width= 1.5, red, ->] (0,-10)--(v7);
					\draw[line width= 1.5, red, ->] (0,-10)--(15,5);
					\draw[line width= 1.5, red, ->] (0,-10)--(4,10);
					\draw[line width= 1.5, red, ->] (0,-10)--(5,10);
			\end{tikzpicture} 
	\caption{Possible combinations of first and second feasible maximal circuit steps from $\vev^{(1)}$.}
	\label{Fig: first two steps CD_efmr}
	\end{figure}
\end{proof}

The following lemma tells us that for an example for $\CD_{fr}>\CD_{f}$ we need a polytope in dimension at least four. In Lemma \ref{ex: with repeated is stronger} we show that such a polytope indeed exists.
\begin{lemma}
For $n\leq 3$, every optimal feasible circuit walk is non-repetitive. Hence the distances $\CD_{f}$ and $\CD_{fr}$ coincide in this case.
\end{lemma}
\begin{proof}
Clearly repetitive circuit walks have length at least three. In case any optimal circuit walk between two vertices is repetitive, any feasible non-repetitive circuit walk must have length at least four. But in dimension $\leq 3$ there always is a such a circuit walk of length at most three by Lemmas \ref{lem:caratheodory} and \ref{lem:smeansbr}.
\end{proof}

\begin{lemma}[$\CD_{fr(b)}>\CD_{f(b)}$]\label{ex: with repeated is stronger}
For $n=4$, there is a polytope with a pair of vertices for which every optimal feasible circuit walk is repetitive (but non-backwards). Hence the distances $\CD_{f}$ and $\CD_{fr}$ and the distances $\CD_{fb}$ and $\CD_{frb}$ differ in this case.
\end{lemma}
\begin{proof}
Let the polytope 
	\[
			P=\left\{\ \vex \in \R^4:\ \vel\leq A\vex\leq \veu\ \right\}
	\]
be defined by
	\[
		A=\left(\begin{array}{cccc}
			1 & 0 & 0 & 0 \\
			0 & 1 & 0 & 0 \\
			0 & 0 & 1 & 0 \\
			0 & 0 & 0 & 1 \\
			1 & 1 & 0 & 0 \\
			1 & 0 & 1 & 0 \\
			1 & 0 & 0 & 1 		
		\end{array}\right),\;
		\vel=\left(\begin{array}{c}
			0 \\ 0 \\ 0 \\ 0 \\ -\infty \\ -\infty \\ - \infty
		\end{array}\right),\;
		\veu=\left(\begin{array}{c}
			3/2 \\ 1 \\ 1 \\ 1 \\ 2 \\ 2 \\ 2
		\end{array}\right) \; .
	\]
The rows of the matrix $A$ define directions of $11$ hyperplanes bounding the polytope.

The \emph{vertices} are the intersections of four of these hyperplanes, in case this intersection is a single point that is contained in $P$. A simple computation shows that we have $23$ vertices $\left(\{0,1\}^4 \backslash \{(1,0,0,0)^T\}\right) \; \cup \; \left(\left\{\frac 32 \right\} \times \left\{0,\frac 12\right\}^3\right)$. In particular $\vev^{(1)}:=(0,0,0,0)^T$ and $\vev^{(2)}:=(1,1,1,1)^T$ are vertices of $P$.

The \emph{circuits} are the potential edge directions for varying $\vel$ and $\veu$, that is, they are given by the intersection of three hyperplanes, in case this intersection is 1-dimensional. Again it is not hard to compute that these directions are $\pm \vece^i$ (where $\vece^{i}$ is the $i$'th unit vector) and $\pm \left(\{1\}\times \{0,-1\}^3 \right)$ and hence constitute the set of circuits.
\bigskip

\noindent \underline{Claim}:\; Every feasible circuit walk from $\vev^{(1)}=(0,0,0,0)^T$ to $\vev^{(2)}=(1,1,1,1)^T$ of length at most three is repetitive.  
\smallskip

\noindent \textit{Proof of claim:}\;  We investigate how we can reach $\vev^{(2)}=(1,1,1,1)^T$ in at most three feasible circuit steps. Observe that in particular we cannot apply circuits that violate the lower bounds of $0$ or the upper bounds of $\frac 32$ (respectively $1$).
Hence as a first feasible circuit step, we can only apply $\vece^1$ or w.l.o.g.{} $\vece^2$. 

Applying $\vece^2$ yields a point $(0,x_2,0,0)^T$ with $x_2\leq 1$. In the second step we can either apply w.l.o.g.{} $\vece^3$, giving $(0,x_2, x_3,0)^T$ with $x_2,x_3\leq 1$, or we apply $\vece^1$ or $(1,-1,0,0)^T$ giving $(x_1,x_2,0,0)^T$ with $x_1\leq \frac 32,\ x_2\leq 1$. From neither of this points we can go to $\vev^{(2)}=(1,1,1,1)^T$ with one more circuit step: We cannot increase the first and the last component at the same time, nor increase the last two components by one simultaneously without decreasing the first component to $\leq \frac 12$.

Applying $\vece^1$ as a first step yields a point $(x_1,0,0,0)^T$ with $x_1 \leq \frac 32$. In the next step we can either increase only one component w.l.o.g.{} only the second one (by applying $\vece^2$ or $(-1,1,0,0)^T$), giving $(x_1,x_2,0,0)^T$ with $x_1\leq \frac 32,\ x_2\leq 1$, but as before we cannot reach $\vev^{(2)}$ in one more circuit step. Otherwise in the second step we increase at least two components (by applying $(-1,1,1,1)^T$ or w.l.o.g.{} $(-1,1,1,0)^T$), giving $(x_1,x_2,x_3,x_4)^T$ with $x_1+x_2=x_1+x_3=x_1+x_4\leq \frac 32,\ x_2,x_3,x_4\leq 1$  (respectively $(x_1,x_2,x_3,0)^T$ with $x_1+x_2=x_1+x_3\leq \frac 32,\ x_2,x_3\leq 1$). In particular we know that $x_2=x_3<1$ or $x_1<1$. Hence to reach to $\vev^{(2)}=(1,1,1,1)^T$ in one more step, we have to increase  the second and third component simultaneously to $1$ (which decreases the first component to $\leq \frac 12$),  or we have to increase the first one without decreasing any other component (that is, we apply $\vece^1$ \emph{again}). This proves our claim.
\bigskip

On the other hand, applying the circuits $\vece^1$, $(-1,1,1,1)^T$ and $\vece^1$ with step length one each is indeed a feasible non-backwards circuit walk of length three from $\vev^{(1)}$ to $\vev^{(2)}$.
\end{proof}

\begin{lemma}[$\CD_{fs}>\CD_{fbr}$]\label{Ex: sign pattern compatible is stronger}
For $n=3$, there is a polytope with a pair of vertices for which every optimal feasible circuit walk is not sign-compatible, and there is such a walk that is non-repetitive and non-backwards. Hence the distance $\CD_{fs}$ differs from $\CD_{fbr}$, $\CD_{fr}$, $\CD_{fb}$, and $\CD_{f}$.
\end{lemma}

\begin{proof}
Consider the polytope 
	\[
			P=\left\{\ \vex \in \R^3:\ \vel\leq B\vex\leq \veu\ \right\}
	\]
defined by
	\[
		B=\left(\begin{array}{ccc}
			1 & 0 & 0  \\
			0 & 1 & 0 \\
			0 & 0 & 1 \\
			1 & 1 & 0 \\
			1 & 0 & 1 \\
		\end{array}\right),\;
		\vel=\left(\begin{array}{c}
			0 \\ 0 \\ 0 \\ -\infty \\ - \infty
		\end{array}\right),\;
		\veu=\left(\begin{array}{c}
			\infty \\ 1 \\ 1 \\ 2\\2 
		\end{array}\right) \; .
	\]
All possible edge directions $\veg$ of $P$ are given by 
	\[
			\pm \left(\begin{array}{r} 1\\0\\0		\end{array}\right),\;
			\pm \left(\begin{array}{r} 0\\1\\0		\end{array}\right),\;
			\pm \left(\begin{array}{r} 0\\0\\1		\end{array}\right),\;
			\pm \left(\begin{array}{r} 1\\-1\\0		\end{array}\right),\;
			\pm \left(\begin{array}{r} 1\\0\\-1	\end{array}\right),\;
			\pm \left(\begin{array}{r} 1\\-1\\-1	\end{array}\right)\; ,
	\]
and the corresponding vectors $B\veg$ are 
	\[
			\pm \left(\begin{array}{r} 1\\0\\0 \\1\\1		\end{array}\right),\;
			\pm \left(\begin{array}{r} 0\\1\\0 \\1\\0		\end{array}\right),\;
			\pm \left(\begin{array}{r} 0\\0\\1 \\0\\1		\end{array}\right),\;
			\pm \left(\begin{array}{r} 1\\-1\\0 \\0\\1		\end{array}\right),\;
			\pm \left(\begin{array}{r} 1\\0\\-1 \\1\\0		\end{array}\right),\;
			\pm \left(\begin{array}{r} 1\\-1\\-1 \\0\\0		\end{array}\right)\;.
	\]
We want to perform circuit walks from $\vev^{(1)}=(0,0,0)^T$ to $\vev^{(2)}=(1,1,1)^T$. We have $B\left(\vev^{(2)}-\vev^{(1)}\right)=(1,1,1,2,2)^T$. Hence only the unit vectors $\vece^{1}$, $\vece^{2}$ and $\vece^{3}$ can be applied in sign-compatible walks.
Thus an optimal feasible sign-compatible walk from $\vev^{(1)}=(0,0,0)^T$ to $\vev^{(2)}=(1,1,1)^T$ has length at least three, as we have to apply all three unit vectors.
	\begin{figure}[H]
		\centering
			\begin{tikzpicture}[scale=2] 
					\coordinate (v1) at (0,0);
					\coordinate (v2) at (2,0);
					\coordinate (v3) at (1.5,0.5);
					\coordinate (v4) at (0.5,0.5);
					\coordinate (v5) at (0,1);
					\coordinate (v6) at (1,1);
					\coordinate (v7) at (1.5,1.5);
					\coordinate (v8) at (0.5,1.5);
					\node[below] at (v1) {$\vev^{(1)}$};
					\node[above] at (v7) {$\vev^{(2)}$};				
					\draw (v1)--(v2);
					\draw[dashed] (v2)--(v3);
					\draw[dashed] (v3)--(v4);
					\draw[dashed] (v1)--(v4);
					\draw (v1)--(v5);
					\draw (v2)--(v6);
					\draw (v2)--(v7);
					\draw[dashed] (v3)--(v7);
					\draw[dashed] (v4)--(v8);
					\draw (v5)--(v6);
					\draw (v6)--(v7);
					\draw (v7)--(v8);
					\draw (v8)--(v5);
					\draw[line width= 1.5, red, ->] (v1)--(v5);
					\draw[line width= 1.5, red, ->] (v5)--(v6);
					\draw[line width= 1.5, red, ->] (v6)--(v7);
			\end{tikzpicture}
	\caption{A feasible sign-compatible circuit walk of length three.}
	\end{figure}
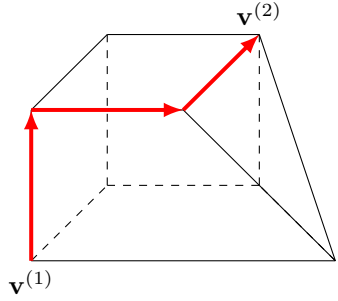
On the other hand, there is a feasible non-repetitive non-backwards circuit walk of length two that is not sign-compatible.
	\begin{figure}[H]
		\centering
			\begin{tikzpicture}[scale=2]
					\coordinate (v1) at (0,0);
					\coordinate (v2) at (2,0);
					\coordinate (v3) at (1.5,0.5);
					\coordinate (v4) at (0.5,0.5);
					\coordinate (v5) at (0,1);
					\coordinate (v6) at (1,1);
					\coordinate (v7) at (1.5,1.5);
					\coordinate (v8) at (0.5,1.5);
					\node[below] at (v1) {$\vev^{(1)}$};
					\node[above] at (v7) {$\vev^{(2)}$};				
					\draw (v1)--(v2);
					\draw[dashed] (v2)--(v3);
					\draw[dashed] (v3)--(v4);
					\draw[dashed] (v1)--(v4);
					\draw (v1)--(v5);
					\draw (v2)--(v6);
					\draw (v2)--(v7);
					\draw[dashed] (v3)--(v7);
					\draw[dashed] (v4)--(v8);
					\draw (v5)--(v6);
					\draw (v6)--(v7);
					\draw (v7)--(v8);
					\draw (v8)--(v5);
					\draw[line width= 1.5, red, ->] (v2)--(v7);
					\draw[line width= 1.5, red, ->] (v1)--(v2);
			\end{tikzpicture}
	\caption{A feasible not sign-compatible circuit walk of length two.}
	\end{figure}
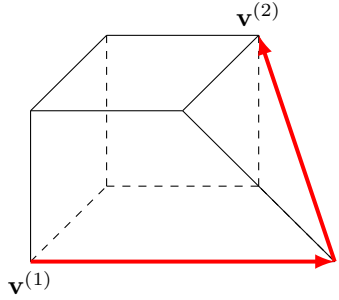
\end{proof}

The following lemma tells us that for an example for $\CD_{fb}>\CD_{f}$ we need a polytope in dimension at least five.
\begin{lemma}\label{ex: 5dim}
For $n \leq 4$, every optimal feasible circuit walk is non-backwards. Hence the distances $\CD_{f}$ and $\CD_{fb}$ coincide in this case.
\end{lemma}

\begin{proof}
We first show that if an optimal feasible circuit walk is backwards then it has length at least four. Clearly it has length at least three. Assume there is a polytope with vertices $\vev^{(1)}$ and $\vev^{(2)}$ that are connected by a feasible circuit walk 
	\[
		\vev^{(1)}=\vey^{(0)},\quad
		\vey^{(1)}=\vey^{(0)}+\alpha_1 \veg^{1}, \quad
		\vey^{(2)}=\vey^{(1)}+\alpha_2 \veg^{2}, \quad
		\vey^{(3)}=\vey^{(2)}+\alpha_3 (-\veg^{1})=\vev^{(2)}
	\]
that is, the walk is backwards. But then there is a feasible circuit walk from $\vev^{(1)}$ to $\vev^{(2)}$ of length two,
	\[
		\vev^{(1)}=\vey^{(0)},\quad
		\bar\vey^{(1)}=\vey^{(0)}+\left(\alpha_1-\alpha_3\right) \veg^{1}, \quad
		\bar\vey^{(2)}=\bar\vey^{(1)}+\alpha_2 \veg^{2} 
		\qquad \text{if }\alpha_1\geq \alpha_3\; ,
	\]
respectively
	\[
		\vev^{(1)}=\vey^{(0)}, \quad
		\bar\vey^{(1)}=\vey^{(0)}+\alpha_2 \veg^{2}, \quad
		\bar\vey^{(2)}=\bar\vey^{(1)}+\left(\alpha_3-\alpha_1\right)(-\veg^{1})
		\qquad \text{if }\alpha_1<\alpha_3 \; .
	\]
Clearly these circuit walks satisfy $\bar\vey^{(2)}=\vev^{(2)}$ and are indeed feasible by convexity of the polytope. Therefore a feasible backwards circuit walk of length three cannot be optimal.

Now in case any optimal circuit walk between two vertices is backwards, any feasible non-backwards circuit walk must have length at least five. But Lemmas \ref{lem:caratheodory} and \ref{lem:smeansbr} imply that for $n\leq 4$ there always is a feasible non-backwards circuit walk of length at most four.
\end{proof}


\section{Diameter hierarchy in dimension two}\label{sec:twodim}
We conclude this paper with a discussion of the different notions of circuit distances in dimension $n=2$. It is easy to see that in this situation the graph diameter is given by $\left\lfloor \frac{k}{2}\right\rfloor$, where $k$ is the number of vertices of the polygon. In particular this number tells us which values $\CD^2_{efm}$ can take. In this section we prove Theorem \ref{Thm: circuit distances in dim 2} that states the possible ranges of all the notions of circuit distances, and  tells us which distance categories coincide for $n=2$ and which remain different. Finally, we will exhibit that $\CD^2_{efm(b)(r)}$ and $\CD^2_{fm(b)(r)}$ can differ significantly in Lemma \ref{lem: distances differ in dim 2}.

\medskip

\noindent\textit{Proof of Theorem \ref{Thm: circuit distances in dim 2}.}
Note that Lemma \ref{ex: edge-directions is stronger}, Lemma \ref{ex: arbitrary length steps is stronger}, Lemma \ref{ex: backwards is stronger} and Lemma \ref{ex: repeated is stronger} show the inequality of the corresponding circuit distances in dimension two and hence also prove the strict inequalities in the circuit hierarchy in Figure \ref{Fig: awesome circuit hierarchy dim 2}. Again the numbers near the inequality symbols refer to these lemmas.

A polygon on $k$ vertices and a pair of vertices with $\CD^2_{efmb}\left(\vev^{(1)},\vev^{(2)}\right)=k-3$ is readily derived from the one given in Figure \ref{Fig: CD0 backwards steps} in Lemma \ref{Ex: CD0 backwards steps} by putting $k-4$ vertices `to the left' of $\vev^{(1)}$ and $\vev^{(2)}$.

In dimension two there are no repetitive edge walks and hence $\CD^2_{efm}=\CD^2_{efmr}$. The claimed range of values for distances of vertices is obvious.

For $\CD^2_{fm(b)(r)}$ we only have to show that there are indeed vertices with feasible maximal circuit distance $ \left\lfloor \frac{k}{2}\right\rfloor$. Lemma \ref{thm: linear CD_fm} proves this for even $k$ and can easily be extended to odd $k$ by adding another vertex.

For
\[\CD^2_{f}=\CD^2_{fr}=\CD^2_{fb}=\CD^2_{fbr}=\CD^2_{fs}=\CD^2_{}\in \{1,2\}\]
it is enough to recall Lemma \ref{lem:caratheodory}. 
\hfill $\square$ \vspace{5pt}

To complete the proof of Theorem \ref{Thm: circuit distances in dim 2}, we still have to show that there are polygons with vertices that have feasible maximal circuit distance $ \left\lfloor \frac{k}{2}\right\rfloor$. For the sake of a clean presentation, we provide the proof for $k$ even. It can readily be extended to the general case.

\begin{lemma}\label{thm: linear CD_fm}
Let $k$ be even. Then there is a polygon on $k$ vertices with  diameter $\frac{k}{2}$ with respect to $\CD_{fm}^2$ .
\end{lemma}

\begin{proof}
Let $k$ even be given. We construct a polygon on vertices $\vev^{(0)},\ldots,\vev^{(k-1)}$ with edges $\left(\vev^{(i)},\vev^{(i+1)}\right)=e_i$ for $i=0,\ldots, k-1$ (where $\vev^{(k)}:=\vev^{(0)}$) such that $\CD^2_{fm}\left(\vev^{(0)}, \vev^{(\frac k2)}\right)=\frac k2$. The corresponding edge walk will be an optimal maximal circuit walk. 

Note that for $n=2$, the edges are the facets of the polygon. Thus there is a direct correspondence of the circuits and the edges as `edge directions'. We will exploit this for a simpler wording in the following, talking about `walking along edges' or `in direction of an edge'.

First of all we fix the edge directions and hence the set of circuits associated with $P$. To this end choose $\frac k2$ slopes $0>s_0>s_1>s_2>\ldots>s_{\frac{k}{2}-1}$ arbitrarily. In the upcoming construction we assign edge $e_0$ slope $-s_0$; edge $e_{k-1}$ slope $s_0$, and for $i=1,\ldots, \frac k2 -1$ we assign $e_i$ slope $s_i$ and $e_{k-1-i}$ slope $-s_i$. This will produce a polygon of shape as depicted in Figure \ref{Fig:shapeofconstruction}.
	\begin{figure}
		\centering
			\begin{tikzpicture}[scale=0.27]
					\coordinate (v0) at (0,0);
					\coordinate (v1) at (10,10);
					\coordinate (v7) at (10,-10);
					\coordinate (v2) at (15,5);
					\coordinate (v6) at (15,-5);
					\coordinate (v3) at (16.6,1.75);
					\coordinate (v5) at (16.6,-1.75);
					\coordinate (v4) at (16.9,0);

					\draw [fill, black] (v1) circle (0.05cm);
					\draw [fill, black] (v2) circle (0.05cm);
					\draw [fill, black] (v3) circle (0.05cm);
					\draw [fill, black] (v4) circle (0.05cm);
					\draw [fill, black] (v5) circle (0.05cm);
					\draw [fill, black] (v6) circle (0.05cm);
					\draw [fill, black] (v7) circle (0.05cm);
					\draw [fill, black] (v0) circle (0.05cm);
					\node[above] at (v1) {$\vev^{(1)}$};
					\node[above right] at (v2) {$\vev^{(2)}$};
					\node[right] at (v3) {$\vev^{(3)}$};
					\node[right] at (v4) {$\vev^{(4)}=\vev^{(\frac k2)}$};
					\node[right] at (v5) {$\vev^{(5)}$};
					\node[below right] at (v6) {$\vev^{(6)}$};
					\node[below ] at (v7) {$\vev^{(7)}$};
					\node[left] at (v0) {$\vev^{(0)}$};

					\draw[black] (v0) node[above right] {$e_0$}--(v1) node[below right] {$e_1$}--(v2) node[below right] {$e_2$}--(v3)node[below left] {$e_3$}--(v4)--(v5)node[above left] {$e_4$}--(v6)node[above right] { $e_5$}--(v7)node[above right] {$e_6$}--(v0)node[below right] {$e_7$};	
			\end{tikzpicture} 
		\caption{Sketch of the polygon for $k=8$.}\label{Fig:shapeofconstruction}
	\end{figure}
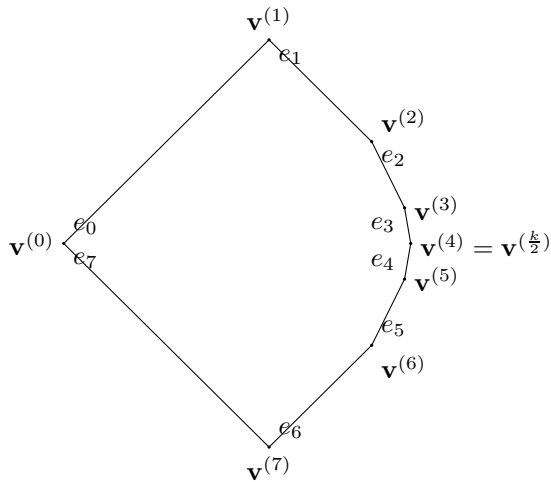	
Observe that the slopes of the edges on an edge walk from $\vev^{(1)}$ to $\vev^{(\frac k2)}$ iteratively become less, just as the slopes of the edges from $\vev^{(k-1)}$ to $\vev^{(\frac k2)}$ become steeper. Further the polygon is symmetric with respect to the first coordinate axis (which we call $x_1$-axis from now on).

It remains to arrange the vertices. We do this iteratively, fixing a pair of vertices $\vev^{(i)}, \vev^{(k-i)}$ in each step such that in $P$ the following property \textbf{(*)} is satisfied:
\medskip

\begin{tabular}{cc}
\textbf{(*)}& \begin{minipage}{12cm} Every maximal feasible circuit walk starting at $\vev^{(0)}$ of length at most $\frac k 2 $ that contains a point $\vev'$ with larger $x_1$-coordinate than $\vev^{(i)}$ or $\vev^{(k-i)}$ (or equivalently, that hits an edge $e_j$ with $i\leq j \leq k-1-j$) must contain  $\vev^{(i)}$ or $\vev^{(k-i)}$.\end{minipage}
\end{tabular}
\medskip

This will immediately imply that the circuit distance $\CD^2_{fm}$ from $\vev^{(0)}$ to $\vev^{(\frac k2 )}$ is $\frac k2$: Every circuit walk of length at most $\frac k 2$ from $\vev^{(0)}$ to $\vev^{(\frac k2 )}$ does reach a $\vev'$ with larger $x_1$-coordinate than every $\vev^{(i)}$ for all $i=1,\ldots,\frac k2 -1$. (We will informally call this `going beyond $\vev^{(i)}$' from now on.) Hence by \textbf{(*)}, any such circuit walk must contain a vertex from each of these $\frac k2-1$ pairs of vertices. This takes at least $\frac k2-1$ steps, and it takes one additional step to reach the target vertex $\vev^{(\frac k2 )}$.
\bigskip

\emph{Construction of initial vertices:} Fix $\vev^{(0)}=(0,0)$. Let edges $e_0$, respectively $e_{k-1}$, start at $\vev^{(0)}$ and end in a $\vev^{(1)}$ on $e_0$ and a $\vev^{(k-1)}$ on $e_{k-1}$ such that $\vev^{(1)}$ and $\vev^{(k-1)}$ have identical $x_1$-coordinates.

\textbf{(*)} holds for the pair $\vev^{(1)},\vev^{(k-1)}$: At $\vev^{(0)}$ we can only apply circuit steps with directions $e_0$ or $e_{k-1}$ (any other direction is too steep). As we apply maximal steps, the second point of any circuit walk is either $\vev^{(1)}$ or $\vev^{(k-1)}$.
\medskip

\emph{Construction of a pair of vertices:} Let the vertices $\vev^{(0)},\vev^{(1)},\ldots,\vev^{(i)},\vev^{(k-i)},\ldots \vev^{(k-1)}$, $i<\frac k2-1$, be constructed and satisfy \textbf{(*)}.
We now construct the vertices $\vev^{(i+1)},\vev^{(k-i-1)}$ together with the incident edges $e_i, e_{k-1-i}$.

\begin{enumerate}
	\item Let edges with directions $e_{i}$  (respectively $e_{k-1-i}$) start at $\vev^{(i)}$ (respectively $\vev^{(k-i)}$). Let $\vew^{(i)}$ be their intersection (which has $x_2$-coordinate $0$). This defines a polygon $P_i$.

	\begin{figure}[H]
		\centering
			\begin{tikzpicture}[scale=0.27]
					\coordinate (v0) at (0,0);
					\coordinate (v1) at (10,10);
					\coordinate (v7) at (10,-10);
					\coordinate (v2) at (15,5);
					\coordinate (v6) at (15,-5);
					\coordinate (v3) at (16.6,1.75);
					\coordinate (v5) at (16.6,-1.75);
					\coordinate (v4) at (16.9,0);
					\coordinate (w) at (17.5,0);

					\draw [fill, black] (v1) circle (0.05cm);
					\draw [fill, black] (v2) circle (0.05cm);
					\draw [fill, black] (v7) circle (0.05cm);
					\draw [fill, black] (v6) circle (0.05cm);
					\draw [fill, black] (w) circle (0.05cm);
					\draw [fill, black] (v0) circle (0.05cm);
					\node[above] at (v1) {$\vev^{(1)}$};
					\node[above right] at (v2) {$\vev^{(2)}$};
					\node[below] at (v7) {$\vev^{(7)}$};
					\node[below right] at (v6) {$\vev^{(6)}$};
					\node[right] at (w) {$\vew^{(2)}$};
					\node[left] at (v0) {$\vev^{(0)}$};

					\draw[black] (v0) node[above right] {$e_0$}--(v1) node[below right] {$e_1$}--(v2) node[below right] {$e_2$}--(w)--(v6)node[above right] { $e_5$}--(v7)node[above right] {$e_6$}--(v0)node[below right] {$e_7$};	
			\end{tikzpicture} 
	\caption{The polygon $P_2$ for $k=8$.}
	\end{figure}
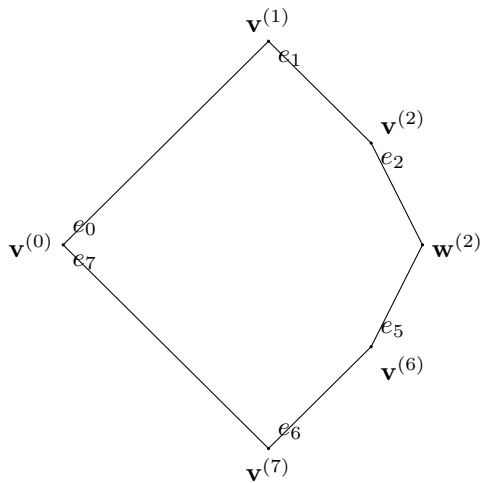	
	
		\item In $P_i$ consider all feasible maximal circuit walks of length at most $\frac k2$ that begin at $\vev^{(0)}$ and do not walk along the (actual) edge $e_{i}$ (respectively $e_{k-1-i}$) to the vertex $\vew^{(i)}$. Then none of these walks contains $\vew^{(i)}$: A step that hits $\vew^{(i)}$ is not allowed to go along the edges we just inserted by definition and we cannot reach $\vew^{(i)}$ from $e_0$ (respectively $e_{k-1}$) in one circuit step by construction. Hence it must start at an edge $e_j$ with w.l.o.g.{} $0<j<i$. But then there would have been a feasible maximal circuit walk of length at most $\frac k2$ in $P_j$ that goes beyond $\vev^{(j)}$ which contradicts the definition of $\vev^{(j)}$ in $P_j$. 
		\item Among all points contained in all of these circuit walks, there are points that have a largest $x_1$-value. These points lie on the edges $e_{i}$ (respectively $e_{k-1-i}$) by construction. We now set $\vev^{(i+1)}$ to be such a point on $e_{i}$ (respectively $\vev^{(k-1-i)}$ on $e_{k-1-i}$). This yields a pair of vertices $\vev^{(i+1)},\vev^{(k-1-i)}$ of identical $x_1$-value and with $\vev^{(i+1)},\vev^{(k-1-i)}\neq \vew^{(i)}$ (with the same arguments as before).
\end{enumerate}

We have to show that $\vev^{(i+1)}$ and $\vev^{(k-i-1)}$ satisfy \textbf{(*)} in $P$.
Therefore, consider a maximal feasible circuit walk in $P$ starting at $\vev^{(0)}$ and of length at most $\frac k 2 $ that goes beyond $\vev^{(i+1)}$. This walk in $P$ translates to a walk in $P_i$ and clearly these walks in $P$ and $P_i$ coincide until they go beyond $\vev^{(i+1)}$ (in both $P$ and $P_i$) by applying some circuit $\veg^{j}$ at some point $\vey^{(j)}$ in the respective circuit walks. Let $\vey^{(j+1)}$ be the subsequent point in the circuit walk in $P$, respectively $\bar\vey^{(j+1)}$ in $P_i$. In particular these $\vey^{(j+1)}$ and $\bar\vey^{(j+1)}$ have a larger $x_1$-value than $\vev^{(i+1)}$. By construction of $\vev^{(i+1)}$ we can only go beyond $\vev^{(i+1)}$ in at most $\frac k2$ circuit steps in $P_i$  when going along the (actual) edge w.l.o.g.{} $e_i$. Hence w.l.o.g.{} $\veg^{j}$ is the edge direction $e_i$ and $\vey^{(j)}\in e_i$. Thus we have  $\vey^{(j)}=\vev^{(i+1)}$ as we apply maximal steps, in particular the vertex $\vev^{(i+1)}$ is contained in the circuit walk in $P$.
\medskip 

\emph{Construction of target vertex:} Set $\vev^{(\frac k2)}:= \vew^{(i)}$ for $i=\frac k2-1$. This concludes the construction of a polygon $P$ with property \textbf{(*)}.
\end{proof}

Theorem \ref{Thm: circuit distances in dim 2} tells us that $\CD^2_{(f(b)(r)(s))}$ is constant always one or two, while the other distances can be linear in the number of vertices, in particular $\CD^2_{fmb}$. It remains to investigate how $\CD^2_{efm(b)(r)}$ and $\CD^2_{fm(b)(r)}$ are related to each other. We conclude by demonstrating that there are polygons for which the former grows linear in the number of vertices while the latter remains constant.

\begin{lemma}\label{lem: distances differ in dim 2}
Let $P$ be a regular polygon on $k$ vertices. Then the diameter with respect to $\CD^2_{efm(b)(r)}$ is given by 
	$
		\begin{cases}\frac{k-1}{2} & \text{if $k$  odd} \\ \frac{k}{2}&  \text{if $k$ even}\end{cases} $\; ,
and the diameter with respect to $\CD^2_{fm(b)(r)}$ is given by
		$\begin{cases}1 &  \text{if $k$ odd} \\ 2&  \text{if $k$ even}\end{cases}$.
\end{lemma}
\begin{proof}
For $\CD^2_{efm(b)(r)}$ the claim is obvious.

To determine the circuit distances $\CD_{fm(b)(r)}$, let $\vev^{(1)},\ldots,\vev^{(k)}$ ($\vev^{(k+1)}:=\vev^{(1)}$) be the vertices of the polygon and $(\vev^{(i)}, \vev^{(i+1)})$ its edges.
Let $k$ be odd. It suffices to show that from $\vev^{(1)}$ we can reach any other vertex in just a single circuit step. For this, it is enough to see $\vev^{(2i)}= \vev^{(1)}+ \alpha \cdot  (\vev^{(i+1)}- \vev^{(i)})$ for some $\alpha$ and $\vev^{(2i+1)}= \vev^{(1)}+ \alpha' \cdot (\vev^{(\frac{k+1}{2}+i)}- \vev^{(\frac{k+1}{2}+i+1)})$ for some $\alpha'$.

	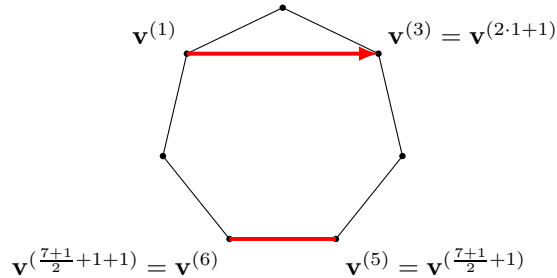
\begin{figure}[H]
		\centering
			\begin{tikzpicture}[scale=0.7]
\useasboundingbox (-5,-6) rectangle (5,1.2);
					\coordinate (v5) at (1,-4.38);
					\coordinate (v4) at (2.25,-2.81);
					\coordinate (v3) at (1.8,-0.87);
					\coordinate (v2) at (0,0); 
					\coordinate (v1) at (-1.8,-0.87);
					\coordinate (v7) at (-2.25,-2.81);
					\coordinate (v6) at (-1,-4.38);
					\draw[black] (v1)--(v2)--(v3)--(v4)--(v5)--(v6)--(v7)--(v1);
					\node[above left] at (v1) {$\vev^{(1)}$};
					\node[above right] at (v3) {$\vev^{(3)}=\vev^{(2\cdot 1+1)}$};
					\node[below right] at (v5) {$\vev^{(5)}=\vev^{(\frac{7+1}{2} +1)}$};
					\node[below left] at (v6) {$\vev^{(\frac{7+1}{2} +1+1)}=\vev^{(6)}$};
					\draw [fill, black] (v1) circle (0.05cm);
					\draw [fill, black] (v2) circle (0.05cm);
					\draw [fill, black] (v3) circle (0.05cm);
					\draw [fill, black] (v4) circle (0.05cm);
					\draw [fill, black] (v5) circle (0.05cm);
					\draw [fill, black] (v6) circle (0.05cm);
					\draw [fill, black] (v7) circle (0.05cm); 
					\draw [line width = 1.5, red, ->] (v1)--(v3);
					\draw [line width = 1.5, red] (v6)--(v5);
			\end{tikzpicture}
	\caption{A circuit step from $\vev^{(1)}$ to $\vev^{(3)}$ and the edge of corresponding direction.}
	\end{figure}

Now let $k$ be even. First observe that there are always two collinear edges and hence not all pairs of vertices are connected by a single circuit step, i.e{.} the diameter cannot be equal to one. As before, we have $\vev^{(2i)}= \vev^{(1)}+ \alpha \cdot  (\vev^{(i+1)}- \vev^{(i)})$ for some $\alpha$. In case we want to walk from $\vev^{(1)}$ to $\vev^{(2i+1)}$ we first go to $\vev^{(2i)}$ and then along edge $(\vev^{(2i)},\vev^{(2i+1)})$, as depicted in the walk from $\vev^{(1)}$ to $\vev^{(5)}$ in Figure \ref{Fig:walkeven}. Hence the regular $k$-polygon for $k$ even has diameter $2$ with respect to  $\CD^2_{fm(b)(r)}$.

	
	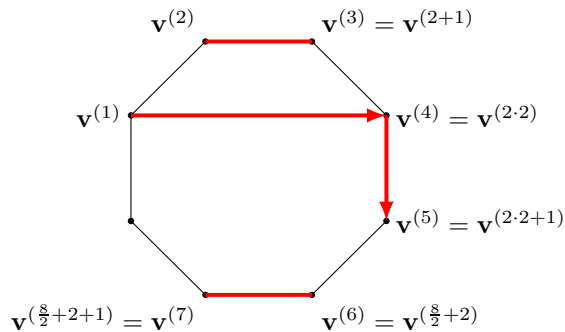
\begin{figure}[H]
		\centering
			\begin{tikzpicture}[scale=0.7]
\useasboundingbox (-5,-6) rectangle (5,1.2);
					\coordinate (v2) at (-1,0);
					\coordinate (v3) at (1,0);
					\coordinate (v4) at (2.4,-1.4);
					\coordinate (v5) at (2.4,-3.4); 
					\coordinate (v6) at (1,-4.8);
					\coordinate (v7) at (-1,-4.8);
					\coordinate (v8) at (-2.4,-3.4);
					\coordinate (v1) at (-2.4,-1.4);
					\draw[black] (v1)--(v2)--(v3)--(v4)--(v5)--(v6)--(v7)--(v8)--(v1);
					\draw [fill, black] (v1) circle (0.05cm);
					\draw [fill, black] (v2) circle (0.05cm);
					\draw [fill, black] (v3) circle (0.05cm);
					\draw [fill, black] (v4) circle (0.05cm);
					\draw [fill, black] (v5) circle (0.05cm);
					\draw [fill, black] (v6) circle (0.05cm);
					\draw [fill, black] (v7) circle (0.05cm); 
					\draw [fill, black] (v8) circle (0.05cm); 
					\node[left] at (v1) {$\vev^{(1)}$};
					\node[above left] at (v2) {$\vev^{(2)}$};
					\node[above right] at (v3) {$\vev^{(3)}=\vev^{(2+1)}$};
					\node[right] at (v4) {$\vev^{(4)}=\vev^{(2\cdot 2)}$};
					\node[right] at (v5) {$\vev^{(5)}=\vev^{(2\cdot 2+1)}$};
					\node[below right] at (v6) {$\vev^{(6)}=\vev^{(\frac{8}{2} +2)}$};
					\node[below left] at (v7) {$\vev^{(\frac{8}{2} +2+1)}=\vev^{(7)}$};
					\draw[line width = 1.5, red, ->] (v1)--(v4);
					\draw[line width = 1.5, red, ->] (v4)--(v5);
					\draw[line width = 1.5, red] (v6)--(v7);
					\draw[line width = 1.5, red] (v2)--(v3);
			\end{tikzpicture}
	\caption{Optimal circuit walks from $\vev^{(1)}$ to $\vev^{(4)}$ and $\vev^{(5)}$ and edges of corresponding direction.}
	\end{figure}	\label{Fig:walkeven}
\end{proof}

\section{Acknowledgements:} We are grateful for the comments we receive from Raymond Hemmecke, Jon Lee, and Eddie Kim regarding some of these
constructions. We are also grateful to Jake Miller for his help during the crafting of this paper. The first author was supported by the Humboldt foundation.

\bibliographystyle{plain}
\bibliography{circuitDiameterHierarchy}

\end{document}